\documentclass[preprint]{imsart}

\usepackage[english]{babel}
\usepackage[OT1]{fontenc}
\usepackage[utf8]{inputenc}
\usepackage{lmodern}
\usepackage{calc}

\usepackage{amsmath}
\usepackage{amsfonts}
\usepackage{amssymb}
\usepackage{amsthm}
\usepackage{euscript}
\usepackage{xspace}
\usepackage{bm}
\usepackage{enumerate}
\usepackage{graphicx}
\usepackage{subcaption}

\usepackage[numbers]{natbib}
\usepackage{hyperref}

\startlocaldefs

\newtheorem{definition}{Definition}

\newtheorem{lemma}{Lemma}[section]
\newtheorem{theorem}{Theorem}
\newtheorem{proposition}{Proposition}
\newtheorem{corollary}{Corollary}

\theoremstyle{definition}

\newtheorem{remark}{Remark}

\numberwithin{equation}{section}

\DeclareMathOperator{\LL}{LL}

\newcommand{\Di}{\ensuremath{\mathcal D}\xspace}

\newcommand{\Fi}{\ensuremath{\mathcal F}\xspace}

\newcommand{\Hi}{\ensuremath{\mathcal H}\xspace}
\newcommand{\Ii}{\ensuremath{\mathcal I}\xspace}

\newcommand{\Oi}{\ensuremath{\mathcal O}\xspace}

\newcommand{\Si}{\ensuremath{\mathcal S}\xspace}

\newcommand{\dimH}{\ensuremath{{\dim}_{\text{\normalfont\tiny H}}}\xspace}

\newcommand{\vset}{\ensuremath{{\emptyset}}\xspace}

\newcommand{\pth}[1]{(#1)}
\newcommand{\pthb}[1]{\bigl(#1\bigr)}
\newcommand{\pthB}[1]{\Bigl(#1\Bigr)}
\newcommand{\pthbb}[1]{\biggl(#1\biggr)}

\newcommand{\bkt}[1]{[#1]}
\newcommand{\bktb}[1]{\bigl[#1\bigr]}

\newcommand{\bktbb}[1]{\biggl[#1\biggr]}

\newcommand{\brc}[1]{\{#1\}}
\newcommand{\brcb}[1]{\bigl\{#1\bigr\}}
\newcommand{\brcB}[1]{\Bigl\{#1\Bigr\}}
\newcommand{\brcbb}[1]{\biggl\{#1\biggr\}}

\newcommand{\scpr}[2]{\langle #1,#2 \rangle}
\newcommand{\scprb}[2]{\bigl\langle #1,#2 \bigr\rangle}

\newcommand{\dt}{\ensuremath{\mathrm d}\xspace} 
\newcommand{\eqdef}{:=}

\newcommand{\eqas}{\overset{\mathrm{a.s.}}{=}}

\newcommand{\ivoo}[1]{\ensuremath{(#1)}}
\newcommand{\ivoob}[1]{\ensuremath{\bigl(#1\bigr)}}

\newcommand{\ivof}[1]{\ensuremath{(#1]}}
\newcommand{\ivofb}[1]{\ensuremath{\bigl(#1\bigr]}}

\newcommand{\ivfo}[1]{\ensuremath{[#1)}}
\newcommand{\ivfob}[1]{\ensuremath{\bigl[#1\bigr)}}

\newcommand{\ivff}[1]{\ensuremath{[#1]}}
\newcommand{\ivffb}[1]{\ensuremath{\bigl[#1\bigr]}}

\newcommand{\abs}[1]{\lvert#1\rvert}
\newcommand{\absb}[1]{\bigl\lvert#1\bigr\rvert}

\newcommand{\absbb}[1]{\biggl\lvert#1\biggr\rvert}

\newcommand{\norm}[1]{\lVert#1\rVert}
\newcommand{\normb}[1]{\bigl\lVert#1\bigr\rVert}

\newcommand{\normbb}[1]{\biggl\lVert#1\biggr\rVert}

\newcommand{\floor}[1]{\lfloor#1\rfloor}

\newcommand{\ceil}[1]{\lceil#1\rceil}

\newcommand{\pr}[2][]{\mathbb{P}#1\pth{#2}}
\newcommand{\prb}[2][]{\mathbb{P}#1\pthb{\hspace{1pt}#2\hspace{1pt}}}
\newcommand{\prB}[2][]{\mathbb{P}#1\pthB{#2}}
\newcommand{\prbb}[2][]{\mathbb{P}#1\pthbb{#2}}

\newcommand{\prcb}[3][]{\mathbb{P}#1\pthb{\hspace{1pt}#2\bigm|#3\hspace{1pt}}}
\newcommand{\prcB}[3][]{\mathbb{P}#1\pthB{#2\Bigm|#3}}
\newcommand{\prcbb}[3][]{\mathbb{P}#1\pthbb{#2\biggm|#3}}

\newcommand{\esp}[2][]{\mathbb{E}#1\bkt{#2}}
\newcommand{\espb}[2][]{\mathbb{E}#1\bktb{\hspace{1pt}#2\hspace{1pt}}}

\newcommand{\espc}[3][]{\mathbb{E}#1\bkt{\hspace{1pt}#2\hspace{1.5pt}|\hspace{1.5pt}#3\hspace{1pt}}}

\newcommand{\e}{\ensuremath{\mathrm{e}}\xspace}

\newcommand{\R}{\ensuremath{\mathbf{R}}\xspace}
\newcommand{\Q}{\ensuremath{\mathbf{Q}}\xspace}
\newcommand{\N}{\ensuremath{\mathbf{N}}\xspace}
\newcommand{\Z}{\ensuremath{\mathbf{Z}}\xspace}

\newcommand{\indi}{\ensuremath{\mathbf{1}}\xspace}
\newcommand{\eps}{\varepsilon}
\newcommand{\sbullet}{{\substack{\text{\fontsize{4pt}{4pt}$\bullet$}}}}

\newcommand{\vsp}{\vspace{.15cm}}

\newcommand{\hex}{\hspace{1ex}}

\endlocaldefs

\arxiv{arXiv:1302.3140}

\begin{document}

\begin{frontmatter}

\title{Fine regularity of L\'evy processes and linear (multi)fractional stable motion}
\runtitle{Fine regularity of L\'evy processes and linear (multi)fractional stable motion}
\hypersetup{pdftitle=Fine regularity of L\'evy processes and linear (multi)fractional stable motion}

\author{\fnms{Paul} \snm{Balan\c{c}a}%
\ead[label=e1]{paul.balanca@ecp.fr}%
\ead[label=u1,url]{www.mas.ecp.fr/recherche/equipes/modelisation\_probabiliste}%
}%

\address{
\printead{e1}\\\printead{u1}\\[1em]
\'Ecole Centrale Paris\\
Laboratoire MAS, ECP\\
Grande Voie des Vignes - 92295 Ch\^atenay-Malabry, France
}

\affiliation{\'Ecole Centrale Paris}
\runauthor{Paul Balan\c{c}a}

\begin{abstract}
  In this work, we investigate the fine regularity of L\'evy processes using the 2-microlocal formalism. This framework allows us to refine the multifractal spectrum determined by Jaffard and, in addition, study the oscillating singularities of Lévy processes. The fractal structure of the latter is proved to be more complex than the classic multifractal spectrum and
  is determined in the case of alpha-stable processes. As a consequence of these fine results and the properties of the 2-microlocal frontier, we are also able to completely characterise the multifractal nature of the linear fractional stable motion (extension of fractional Brownian motion to $\alpha$-stable measures) in the case of continuous and unbounded sample paths as well. The regularity of its multifractional extension is also presented, indirectly providing an example of a stochastic process with a non-homogeneous and random multifractal spectrum.
\end{abstract}



\end{frontmatter}


\section{Introduction}

The study of sample path continuity and H\"older regularity of stochastic processes is a very active field of research in probability theory. The existing literature provides a variety of uniform results on local regularity, especially on the modulus of continuity, for rather general classes of random fields (see e.g. \citet{Marcus.Rosen-2006}, \citet{Adler.Taylor-2007} on Gaussian processes and \citet{Xiao-2010} for more recent developments).

On the other hand, the structure of pointwise regularity is generally more complex as the latter often tends to behave erratically as time passes. 
This type of sample path behaviour was first put into light on Brownian motion by \citet{Orey.Taylor-1974} and \citet{Perkins-1983}. 
They respectively studied \emph{fast} and \emph{slow points} which characterize logarithmic variations of the pointwise modulus of continuity, and proved that the sets of times with a given pointwise regularity have a distinct fractal geometry. \citet{Khoshnevisan.Shi-2000} have recently extended this study of fast points to fractional Brownian motion.

Lévy processes with a jump compound also present an interesting pointwise behaviour. Indeed, \citet{Jaffard-1999} has proved that despite the random variations of the pointwise exponent, the level sets of the latter show a specific fractal structure. This seminal work has been enhanced and extended by \citet{Durand-2009}, \citet{Durand.Jaffard-2012} and \citet{Barral.Fournier.ea-2010}. Particularly, the latter have proved that Markov processes have a range of admissible pointwise behaviours wider and richer than Lévy processes. In the aforementioned works, \emph{multifractal analysis} happens to be the key concept to study and characterise the local fluctuations of the pointwise regularity. In order to be more specific, we recall a few definitions.
\begin{definition}[Pointwise exponent]  \label{def:pointwise_holder}
  A function $f:\R\rightarrow\R^d$ belongs to $C^\alpha_t$, where $t\in\R$ and $\alpha>0$, if there exist $C>0$, $\rho>0$ and a polynomial $P_t$ of degree less than $\alpha$ such that
  \[
    \forall u\in B(t,\rho);\quad \norm{ f(u) - P_t(u) } \leq C \abs{t - u}^\alpha.
  \]
  The \emph{pointwise H\"older exponent} of $f$ at $t$ is then defined by $\alpha_{f,t} = \sup\brc{ \alpha\geq 0 : f\in C^\alpha_t }$, where by convention $\sup\brc{\emptyset} = 0$.
\end{definition}
Multifractal analysis is interested in the fractal geometry of the level sets of the pointwise exponent, which are also called the \emph{iso-H\"older sets} of $f$:
\begin{equation}  \label{eq:def_set_Eh}
  E_h = \brcb{ t\in\R : \alpha_{f,t} = h}\quad\text{for every $h\in\R_+\cup\brc{+\infty}$.}
\end{equation}
The geometry of the collection $(E_h)_{h\in\R_+}$ is then studied through its Hausdorff dimension, defining for that purpose the \emph{local spectrum of singularities} $d_f(h,V)$ of $f$:
\begin{equation}  \label{eq:def_spectrum}
  d_f(h,V) = \dimH(E_h \cap V)\quad\text{for every $h\in\R_+\cup\brc{+\infty}$ and $V\in\Oi$,}
\end{equation}
where $\Oi$ designates the collection of nonempty open sets of $\R$ and $\dimH$ is the Hausdorff dimension, with by convention $\dimH(\vset) = -\infty$ (we refer to \cite{Falconer-2003} for the complete definition of the latter).

Even though $(E_h)_{h\in\R_+}$ are random sets, stochastic processes such as L\'evy processes \cite{Jaffard-1999}, L\'evy processes in multifractal time \cite{Barral.Seuret-2007} and fractional Brownian motion have a deterministic multifractal spectrum. Furthermore, these random fields are also said to be \emph{homogeneous} since the quantity $d_X(h,V)$ is independent of the open set $V$ for any $h\in\R_+$. In addition, when the pointwise exponent is constant along sample paths, the spectrum is described as \emph{degenerate}, i.e. its support is reduced to a single point (e.g. the Hurst exponent $H$ in the case of f.B.m.). 
Nevertheless, note that \citet{Barral.Fournier.ea-2010} and \citet{Durand-2008} have provided examples of respectively Markov jump processes and wavelet random series with a non-homogeneous and random spectrum of singularities.
\vsp

As outlined in Equations \eqref{eq:def_set_Eh} and \eqref{eq:def_spectrum}, multifractal analysis usually focuses on the structure of pointwise regularity. Unfortunately, as presented by \citet{Meyer-1998}, the pointwise Hölder exponent suffers of a couple of drawbacks: it lacks of stability under the action of pseudo-differential operators and it is not always characterised by the wavelets coefficients. In addition, several simple deterministic examples such as the \emph{Chirp function} $t\mapsto \abs{t}^\alpha\sin\pthb{\abs{t}^{-\beta}}$ show that it does not fully capture the local geometry and oscillations of a function.

Several approaches, such as the \emph{oscillating, chirp and weak scaling exponents} introduced by \citet{Arneodo.Bacry.ea-1998} and \citet{Meyer-1998}, have emerged in the literature to address the limits of the pointwise exponent and supplement the latter by characterising other aspects of the local regularity. Interestingly, the aforementioned concepts are embraced by a single framework called \emph{2-microlocal analysis}. It has first been introduced by \citet{Bony-1986} in the deterministic frame to study singularities of generalised solutions of PDEs. Several authors have then investigated in \cite{Jaffard-1991,Jaffard.Meyer-1996,Meyer-1998,LevyVehel.Seuret-2004} this framework more deeply, determining in particular the close connection between the 2-microlocal formalism and the previous scaling exponents. More recently, \citet{Herbin.LevyVehel-2009} have developed a stochastic approach of this framework to investigate the fine regularity of stochastic processes such as Gaussian processes, martingales and stochastic integrals. 

Similarly to the pointwise Hölder exponent, the introduction of this formalism starts with the definition of appropriate functional spaces, named \emph{2-microlocal spaces}. We begin with a simpler, but narrower, definition to give an intuition of these concepts.
\begin{definition}  \label{def:2ml_spaces}
  Suppose $t\in\R$, $s'\in\R$ and $\sigma\in\ivoo{0,1}$ such that $\sigma-s'\notin\N$. A function $f:\R\rightarrow\R^d$ belongs to the \emph{2-microlocal space} $C^{\sigma,s'}_t$ if there exist $C>0$, $\rho>0$ and a polynomial $P_t$ such that for all $u,v\in B(t,\rho)$:
  \begin{equation} \label{eq:def_2ml_spaces}
    \normb{ f(u)-P_t(u) - f(v)+P_t(v) } \leq C\abs{u-v}^\sigma \pthb{ \abs{u-t}+\abs{v-t} }^{-s'}.
  \end{equation}
  In addition, $P_t$ is unique if we suppose its degree is smaller than $\sigma-s'$. In this case, it corresponds to the Taylor polynomial of order $\floor{\sigma-s'}$ of $f$ at $t$.
\end{definition}
The 2-microlocal spaces are therefore parametrised by a pair $(s',\sigma)$ of real numbers and we clearly observe on Equation~\eqref{eq:def_2ml_spaces} that they extend the underlying ideas of the classic Hölder spaces. To define these elements for any $\sigma\in\R\setminus\Z$, we need to slightly complexify the form of the increments considered.
\begin{definition}  \label{def:2ml_spaces_gen}
  Suppose $t\in\R$ and $b < t$ is fixed. In addition, consider $s'\in\R$, $\sigma\in\R\setminus\Z$ and $k\in\Z$ such that $\sigma-s'\notin\N$ and $\sigma+k\in\ivoo{0,1}$. A function $f:\R\rightarrow\R^d$ belongs to the \emph{2-microlocal space} $C^{\sigma,s'}_t$ if there exist $C>0$, $\rho>0$ and a polynomial $P_{t,k}$ such that  for all $u,v\in B(t,\rho)$:
  \begin{equation}  \label{eq:2ml_spaces_time_gen}
    \normb{I^{k}_{b+} f(u)-P_{t,k}(u) - I^{k}_{b+} f(v)+P_{t,k}(v)} \leq C\abs{u-v}^{\sigma+k} \pthb{ \abs{u-t}+\abs{v-t} }^{-s'},
  \end{equation}
  where $I^{k}_{b+}f$ designates the derivative of order $-k$ when $k\leq 0$ and the iterated integral of order $k$ when $k> 0$, i.e. $\pthb{I^{k}_{b+} f}(u) \eqdef 1/\Gamma(k-1) \int_b^u (u-s)^{k-1} f(s) \,\dt s$.
\end{definition}
The time-domain characterisation \eqref{eq:def_2ml_spaces}-\eqref{eq:2ml_spaces_time_gen} of 2-microlocal spaces has first been obtained by \citet{Kolwankar.LevyVehel-2002} in the case $\sigma\in\ivoo{0,1}$ and then extended by \citet{Seuret.LevyVehel-2003} and \citet{Echelard-2007} to $\sigma\in\R\setminus\Z$. Note that the previous characterisation does not depend on the value of the constant $b$, since a modification of the latter simply induces an adjustment of the polynomial $P_t$. 

Even though, we restrict ourselves in Definitions~\ref{def:2ml_spaces}-\ref{def:2ml_spaces_gen} to usual functions, 2-microlocal spaces were originally introduced by \citet{Bony-1986} for tempered distributions $\Si'(\R)$. The first definition given by \citet{Bony-1986} relies on the Littlewood--Paley decomposition of distributions, and thereby corresponds to a description in the Fourier space. Another characterisation based on wavelet coefficients has also been presented by \citet{Jaffard-1991}. In addition, note that the previous characterisation is in fact equivalent the \emph{localised 2-microlocal spaces} which are also defined for distributions in $\Di'(\R)$ (we refer to \cite{Meyer-1998} for a more precise distinction between global and local definitions of the 2-microlocal spaces). 

One major property of the 2-microlocal spaces is their stability under the action of pseudo-differential operators. In particular, as proved by \citet[Th 1.1]{Jaffard.Meyer-1996}, they satisfy
\begin{equation} \label{eq:2ml_spaces_int}
  \forall \alpha>0;\quad f\in C^{\sigma,s'}_t \quad\Longleftrightarrow\quad  I_+^\alpha f \in C^{\sigma+\alpha,s'}_t,
\end{equation}
where the fractional integral of $f$ of order $\alpha\geq 0$ is defined by: $\pthb{I_+^\alpha f}(u) \eqdef 1/\Gamma(\alpha)\int_\R (u-s)_+^{\alpha-1} f(s) \,\dt s$. Note that the latter definition of the operator $I^\alpha_+$ coincides with the fractional integral presented in \cite{Jaffard.Meyer-1996} for tempered distributions (we refer to the book of \citet{Samko.Kilbas.ea-1993} for an extensive study of the subject).

Similarly to the pointwise H\"older exponent, the introduction of 2-microlocal spaces leads naturally to the definition of a regularity tool named the \emph{2-microlocal frontier}:
\[
  \forall s'\in\R;\quad \sigma_{f,t}(s') = \sup\brcb{\sigma\in\R : f\in C^{\sigma,s'}_{t}}.
\]
Due to several inclusion properties of the 2-microlocal spaces, the map $s'\mapsto\sigma_{f,t}(s')$ is well-defined and satisfies:
\begin{itemize} \itemsep1pt
  \item $\sigma_{f,t}(\cdot)$ is a concave non-decreasing function;
  \item $\sigma_{f,t}(\cdot)$ has left and right derivatives between $0$ and $1$.
\end{itemize}
Furthermore, as a consequence of Equation~\eqref{eq:2ml_spaces_int}, $\sigma_{f,t}(\cdot)$ is stable under the action of pseudo-differential operators. 
As a function, the 2-microlocal frontier $\sigma_{f,t}(\cdot)$ offers a more complete and richer description of the local regularity and cover in particular the usual Hölder exponents:
\begin{equation*}
  \widetilde\alpha_{f,t} = \sigma_{f,t}(0) \quad\text{and}\quad \alpha_{f,t} = -\inf\brc{s' : \sigma_{f,t}(s')\geq 0},
\end{equation*}
where the last equality has been proved by \citet{Meyer-1998} under the assumption $\omega(h) = \mathrm{O}\,(1/\abs{\log(h)})$ on the modulus of continuity of $f$. Several other scaling exponents previously outlined can also be retrieved from the frontier: the \emph{chirp} and \emph{weak scaling exponents} introduced by \citet{Meyer-1998} are given by:
\begin{equation*}
  \beta^c_{f,t} = \brcbb{ \frac{\dt \sigma_{f,t}}{\dt s'}\biggl\vert_{s'\rightarrow-\infty} }^{\!-1}-1 \quad\text{and}\quad
  \beta^w_{f,t} = \lim_{s'\rightarrow-\infty} \sigma_{f,t}(s') - s';
\end{equation*}
These two elements characterise the asymptotic regularity of a function after a large number of integrations and the latter was been specifically introduced to supplement the pointwise exponent in multifractal analysis. The \emph{oscillating exponent} defined by \citet{Arneodo.Bacry.ea-1998} can also be retrieved from the 2-microlocal frontier:
\begin{equation*}
  \beta^o_{f,t} = \brcbb{ \frac{\dt \sigma_{f,t}}{\dt s'}\biggl\vert_{s' = -\alpha_{f,t\, -}} }^{\!-1}-1.
\end{equation*}
The latter aims to capture the oscillating behaviour by studying the regularity after infinitesimal integrations. Note that the original definition of these exponents are based on Hölder spaces (see \cite{Seuret.Vehel-2002} for an extensive review).

In the stochastic framework, Brownian motion provides a simple example of 2-microlocal frontier: with probability one and for all $t\in\R$
\begin{equation}  \label{eq:2ml_bm}
  \forall s'\in\R;\quad \sigma_{B,t}(s') = \pthB{\frac{1}{2} + s'}\wedge\frac{1}{2}.
\end{equation}
Using the common terminology of \citet{Arneodo.Bacry.ea-1997} and \citet{Meyer-1998}, Brownian motion is said to have \emph{cusp singularities} as $\beta^w_{B,t}=\alpha_{B,t}$ and $\beta^o_{B,t} = 0$,. On the other hand, \emph{oscillating singularities} appear when the slope of the frontier is strictly smaller than $1$ at $s'=-\alpha_{f,t}$, or equivalently, when $\beta^w_{B,t}>\alpha_{B,t}$. This oscillating behaviour is well-illustrated by the chirp function whose frontier and scaling exponents at $0$ respectively are equal to $\sigma_{f,0}(s') = (\alpha+s')(1+\beta)$, $\alpha_{f,0}=0$, $\beta^c_{f,0}=\beta^o_{f,0}=\beta$ and $\beta^w_{f,0}=\infty$.\vsp

In this paper, we combine the 2-microlocal formalism with the classic use of multifractal analysis to obtain a finer and richer description of the regularity of Lévy processes. Following the path of \cite{Jaffard-1999,Durand-2009,Durand.Jaffard-2012}, we extend the multifractal description (Section~\ref{sec:2ml_levy}) to the aforementioned scaling exponents and the 2-microlocal frontier. We present in particular how this formalism allows to capture and describe the oscillating singularities of Lévy processes. The fractal structure of the latter is determined for a few classes of Lévy processes which include alpha-stable processes.

This finer analysis of the sample path properties of L\'evy processes happens to be very useful for the study of another class of processes named linear fractional stable motion (LFSM). The LFSM is a common $\alpha$-stable self-similar process with stationary increments which can be seen as the extension of the fractional Brownian motion to the non-Gaussian frame. In Section~\ref{sec:2ml_lfsm}, we completely characterize the multifractal nature of the LFSM, unifying the geometrical description of the sample paths independently of their boundedness. In addition, we also extend this analysis to the multifractional generalisation of the LFSM.

\subsection{Statement of the main results}

As it is well known, an $\R^d$-valued L\'evy process $(X_t)_{t\in\R_+}$ has stationary and independent increments. Furthermore, its law is determined by the L\'evy--Khintchine formula (see e.g. \cite{Sato-1999}): for all $t\in\R_+$ and $\lambda\in\R^d$, $\esp{e^{i\scpr{\lambda}{X_t}}} = e^{t\psi(\lambda)}$ where $\psi$ is given by
\[
  \forall \lambda\in\R^d;\quad\psi(\lambda) = i\scpr{a}{\lambda} - \frac{1}{2}\scpr{\lambda}{Q\lambda} + \int_{\R^d} \pthb{ e^{ i\scpr{\lambda}{x} } - 1 - i\scpr{\lambda}{x} \indi_{\brc{\norm{x}\leq 1}} } \pi(\dt x).
\]
In the previous expression, $Q$ is a non-negative symmetric matrix and $\pi$ is the L\'evy measure, i.e. a positive Radon measure on $\R^d\setminus\brc{0}$ such that $\int_{\R^d} (1\wedge \norm{x}^2) \,\pi(\dt x) < \infty$. 
Throughout this paper, it will always be assumed that $\pi(\R^d) = +\infty$ since otherwise, the L\'evy process corresponds to the sum of a simple compound Poisson process with drift and a Brownian motion whose regularity is well-known.

Sample path properties of L\'evy processes are known to depend on the growth of the L\'evy measure near the origin. More precisely, \citet{Blumenthal.Getoor-1961} have defined the following exponents $\beta$ and $\beta'$,
\begin{equation} \label{eq:def_beta}
  \beta = \inf\brcbb{ \delta \geq 0 : \int_{\R^d} \pthb{ 1\wedge \norm{x}^\delta } \,\pi(\dt x) < \infty }\quad\text{and}\quad \beta' =
  \begin{cases}
    \beta\ &\text{if } Q = 0; \\
    2  &\text{if }Q\neq 0. 
  \end{cases}
\end{equation}
Owing to $\pi$'s definition, $\beta,\beta'\in\ivff{0,2}$. \citet{Pruitt-1981} proved that $\alpha_{X,0} \eqas 1/\beta$ when $Q=0$. Note that several other exponents have been introduced in the literature to study the sample path properties of L\'evy processes  (see e.g. \cite{Khoshnevisan.Xiao-2002,Khoshnevisan.Shieh.ea-2008} for some recent developments).

\citet{Jaffard-1999} has studied the spectrum of singularities of L\'evy processes under the following assumption on the measure $\pi$,
\begin{equation} \label{eq:hyp_jaffard}
  \sum_{j\in\N} 2^{-j} \sqrt{ C_j \log(1 + C_j) } < \infty,\quad\text{where }\ C_j = \int_{2^{-j-1} < \norm{x} \leq 2^{-j} } \pi(\dt x).
\end{equation}
Under the Hypothesis~\eqref{eq:hyp_jaffard}, Theorem 1 in \cite{Jaffard-1999} states that the multifractal spectrum of a Lévy process $X$ is almost surely equal to 
\begin{equation} \label{eq:spectrum_levy}
  \forall V\in\Oi;\quad d_X(h,V) = 
  \begin{cases}
    \beta h   & \text{ if } h\in\ivfo{0, 1/\beta'}; \\
    1         & \text{ if } h = 1/\beta'; \\
    -\infty   & \text{ if } h\in\ivof{1/\beta',+\infty}.
  \end{cases}
\end{equation}
\citet{Durand-2009} has extended this result to Hausdorff $g$-measures, where $g$ is a gauge function, and \citet{Durand.Jaffard-2012} have generalized the study to multivariate L\'evy fields.\vsp

In this work, we first establish in Proposition \ref{prop:pointwise_levy} a new proof of the multifractal spectrum~\eqref{eq:spectrum_levy} which does not require Assumption~\eqref{eq:hyp_jaffard}. Results obtained by \citet{Durand-2009} on Hausdorff $g$-measure are also indirectly extended using this method.

In order to refine and extend the spectrum of singularities \eqref{eq:spectrum_levy} using the 2-microlocal formalism, we are interested the fractal geometry of the collections of sets $(\widetilde{E}_h)_{h\in\R_+}$ and $(\widehat{E}_h)_{h\in\R_+}$ respectively defined by
\[
  \widetilde{E}_h = \brcb{t\in E_h : \forall s'\in\R;\  \sigma_{X,t}(s') = (h+s')\wedge 0 }\quad\text{ and }\quad \widehat{E}_h = E_h \setminus \widetilde{E}_h.
\]
The introduction of these two collections corresponds to the natural distinction presented in the literature \cite{Arneodo.Bacry.ea-1997,Arneodo.Bacry.ea-1998,Meyer-1998} between two types of singularities: the family $(\widetilde{E}_h)_{h\in\R_+}$ gathers the \emph{cusp singularities} of Lévy processes, i.e. times at which the slope of the 2-microlocal frontier is equal to $1$, whereas the collection $(\widehat{E}_h)_{h\in\R_+}$ regroups the oscillating singularities of the process, i.e. when $\beta^w_{X,t} > \alpha_{X,t}$ and $\beta^o_{X,t} > 0$.

In our first important result, we provide a general description of the fractal geometry of these singularities. 
\begin{theorem}  \label{th:2ml_levy}
  Suppose $X$ is a Lévy process such that $\beta>0$. Then, with probability one, the \emph{cusp singularities} $(\widetilde{E}_h)_{h\in\R_+}$ of $X$ satisfy
  \begin{equation} \label{eq:2ml_levy1}
    \forall V\in\Oi;\quad 
    \dimH(\widetilde{E}_h\cap V) = 
    \begin{cases}
      \beta h   & \text{ if } h\in\ivfo{0, 1/\beta'}; \\
      1         & \text{ if } h = 1/\beta'; \\
      -\infty   & \text{ if } h\in\ivof{1/\beta',+\infty}.
    \end{cases}
  \end{equation}
  Furthermore, the \emph{oscillating singularities} $(\widehat{E}_h)_{h\in\R_+}$ of $X$ are such that
  \begin{equation} \label{eq:2ml_levy2}
    \forall V\in\Oi;\quad
    \dimH(\widehat{E}_h\cap V) \leq
    \begin{cases}
      2\beta h - 1  & \text{ if } h\in\ivoo{1/2\beta, 1/\beta'}; \\
      -\infty   & \text{ if } h\in\ivff{0,1/2\beta}\cup\ivff{1/\beta',+\infty},
    \end{cases}
  \end{equation}
  where the 2-microlocal frontier at $t\in\widehat{E}_h$ verifies $\sigma_{X,t}(s') \leq \pthb{ \tfrac{h+s'}{2\beta h} } \wedge \pthb{\tfrac{1}{\beta'}+s'}\wedge 0$ for all $s'\in\R$.
\end{theorem}
\begin{remark}
Theorem~\ref{th:2ml_levy} induces that $\dimH(\widehat{E}_h) < \dimH(\widetilde{E}_h)$ for every $h\in\ivff{0, 1/\beta'}$. Therefore, in terms of Hausdorff dimension, chirp oscillations that might appear on a Lévy process are always singular compared to the common cusp behaviour. 

We also note that even though sample paths of L\'evy processes do not satisfy the condition $\omega(h)=\text{O}(1/\abs{\log(h)})$ outlined in the introduction, Theorem~\ref{th:2ml_levy} nevertheless ensures that the pointwise H\"older exponent can be retrieved from the 2-microlocal frontier at any $t\in\R_+$ using the formula $\alpha_{X,t} = -\inf\brc{s' : \sigma_{X,t}(s')\geq 0}$. As a consequence, the pointwise regularity of Lévy processes can also be characterised by its wavelet coefficients.
\end{remark}

The determination of the 2-microlocal regularity of Lévy processes allows to deduce the behaviour of several scaling exponents. In particular, we are interested in the multifractal spectrum of the \emph{weak scaling exponent}, whose level sets are defined as:
\begin{equation*}  \label{eq:def_set_Eh_weak}
  E^w_h = \brcb{ t\in\R : \beta^w_{X,t} = h}\quad\text{for every $h\in\R_+\cup\brc{+\infty}$.}
\end{equation*}
\begin{corollary}  \label{cor:2ml_levy_scalings}
  Suppose $X$ is a Lévy process such that $\beta>0$. Then, with probability one
  \begin{equation} \label{eq:2ml_spectrum_levy_weak}
    \forall V\in\Oi; \quad \dimH( E^w_h \cap V) = 
    \begin{cases}
      \beta h   & \text{ if } h\in\ivfo{0,1/\beta'}; \\
      1         & \text{ if } h = 1/\beta'; \\
      -\infty   & \text{ otherwise.}
    \end{cases}
  \end{equation}
  Furthermore, the \emph{oscillating exponent} is such that $\beta^o_{X,t} \leq \max\pthb{0,2\beta h - 1}$ and
  \begin{equation}
    \dimH \brcb{ t\in E_h : \beta^o_{X,t} > 0 } \leq
    \begin{cases}
      2\beta h - 1  & \text{ if } h\in\ivoo{1/2\beta, 1/\beta'}; \\
      -\infty   & \text{ otherwise.}
    \end{cases}
  \end{equation}
  Finally, the \emph{chirp scaling exponent} satisfies $\beta^c_{X,t} = 0$ for all $t\in\R$.
\end{corollary}
According to Corollary~\ref{cor:2ml_levy_scalings}, the multifractal spectrum associated to the weak scaling exponent is the same as the classic one~\eqref{eq:spectrum_levy} despite the oscillating singularities which might exist. We also note that the latter do not influence the chirp scaling exponent, showing that chirp oscillations tend to disappear after multiple integrations.

Following the ideas presented by \citet{Meyer-1998}, it is also natural to investigate geometrical properties of the sets $(E_{\sigma,s'})_{\sigma,s'\in\R}$ defined by
\[
  E_{\sigma,s'} = \brcb{ t\in\R_+ : \forall u'>s';\ X_\sbullet\in C^{\sigma,u'}_t\ \text{ and }\ \forall u'<s';\ X_\sbullet\notin C^{\sigma,u'}_t }.
\]
This collection of sets can be seen as the level sets of the 2-microlocal frontier for a fixed $\sigma$.
\begin{corollary} \label{cor:2ml_levy}
  Suppose $X$ is a Lévy process such that $\beta>0$. Then, with probability one and for all $\sigma\in\R_-$,
  \begin{equation} \label{eq:2ml_spectrum_levy}
    \forall V\in\Oi; \quad \dimH(E_{\sigma,s'} \cap V) = 
    \begin{cases}
      \beta s   & \text{ if } s\in\ivfo{0,1/\beta'}; \\
      1         & \text{ if } s = 1/\beta'; \\
      -\infty   & \text{ otherwise.}
    \end{cases}
  \end{equation}
  where $s$ denotes the common 2-microlocal parameter $s = \sigma - s'$. Furthermore, for all $s'\in\R$, $E_{0,s'} = E_{-s'}$ and $E_{\sigma,s'}$ is empty if $\sigma > 0$. 
\end{corollary}
As for the weak scaling exponent, we obtain in Corollary~\ref{cor:2ml_levy} a multifractal spectrum which takes the same form as Equation~\eqref{eq:spectrum_levy} (note that the latter corresponds to the case $\sigma=0$). In addition, the oscillating singularities are also not captured by these scaling exponents and the spectrum associated.\vsp

Theorem~\ref{th:2ml_levy} provides an upper bound of the Hausdorff dimension of the oscillating singularities of a general Lévy process. In Section~\ref{ssec:2ml_alpha}, we obtain the exact estimates for some specific classes of Lévy processes, proving in particular that the Blumenthal--Getoor exponent does not entirely characterise the structure of these chirp oscillations. 
\begin{proposition}  \label{prop:2ml_levy_oneside}
  Suppose $\pi$ is a Lévy measure on $\R$ such that $\pi(\R_\pm) = 0$ and $X$ is a Lévy process with generating triplet $(a,Q,\pi)$. Then, with probability one, $\widehat{E}_h = \vset$ for all $h\in\R_+$, i.e.
  \[
    \forall t\in\R_+, \ \forall s'\in\R;\quad \sigma_{X,t}(s') = \pthb{\alpha_{X,t} + s'}\wedge 0. 
  \]
\end{proposition}
Note in particular that subordinators do not have oscillating singularities, which is quite understandable because of their monotonicity.

Nevertheless, these singularities might appear as well for rather natural classes of processes such as alpha-stable Lévy processes.
\begin{theorem}  \label{th:2ml_alpha}
  Suppose $X$ is a Lévy process parametrised by $(0,0,\pi)$, where the Lévy measure $\pi$ has the following form
  \begin{align}  \label{eq:levy_alpha_gen}
    \pi(\dt x) = a_1 \,\abs{x}^{-1-\alpha_1} \,\indi_{\R_+}\dt x + a_2 \,\abs{x}^{-1-\alpha_2} \,\indi_{\R_-}\dt x,
  \end{align}
  and  $a_1,a_2 > 0$ and $\alpha_1,\alpha_2\in\ivoo{0,2}$. 

  Then, the Blumenthal--Getoor exponent of $\pi$ is equal to $\beta = \max\pth{\alpha_1,\alpha_2}$ and with probability one, the oscillating singularities of $X$ satisfy
  \begin{equation}  \label{eq:2ml_alpha_gen}
    \forall V\in\Oi;\quad
    \dimH(\widehat{E}_h\cap V) =
    \begin{cases}
      (\alpha_1 + \alpha_2) h - 1  & \text{ if}\hex h\in\ivoob{1/(\alpha_1 + \alpha_2), 1/\beta}; \\
      -\infty        & \text{ otherwise.}
    \end{cases}
  \end{equation}
\end{theorem}
One of the interesting aspects of the previous result is to show that the Hausdorff dimension of the oscillating singularities of Lévy processes is not necessarily governed by the Blumenthal--Getoor exponent, but also takes into account the symmetrical aspect of the Lévy measure. Furthermore, Theorem~\ref{th:2ml_alpha} proves that the upper bound obtained in Theorem~\ref{th:2ml_levy} is optimal, since in the case of an alpha-stable process parametrised by $(\alpha,\beta_\alpha)$, with probability one
\begin{equation}  \label{eq:2ml_alpha}
  \forall V\in\Oi;\quad
  \dimH(\widehat{E}_h\cap V) =
  \begin{cases}
    2\alpha h - 1  & \text{ if } h\in\ivoo{1/2\alpha, 1/\alpha}\text{ and } \beta_\alpha\in\ivoo{-1,1}; \\
    -\infty        & \text{ otherwise.}
  \end{cases}
\end{equation}
Note that owing to Proposition~\ref{prop:2ml_levy_oneside}, an alpha-stable process whose skewness parameter $\beta_\alpha$ is equal to $1$ or $-1$ does not have oscillating singularities.\vsp

The fine 2-microlocal structure presented Theorems~\ref{th:2ml_levy} and \ref{th:2ml_alpha} happens to be interesting outside the scope of L\'evy processes. More precisely, it allows to characterized the multifractal nature of the linear fractional stable motion (LFSM). The latter is a fractional extension of alpha-stable Lévy processes and is usually defined by the following stochastic integral (see e.g. \cite{Samorodnitsky.Taqqu-1994})
\begin{equation} \label{eq:def_lfsm}
  X_t = \int_\R  \brcB{ (t-u)_+^{H-1/\alpha} - (-u)_+^{H-1/\alpha} }  \,M_{\alpha}(\dt u),
\end{equation}
where $M_{\alpha}$ is an alpha-stable random measure parametrised by $\alpha\in\ivoo{0,2}$ and $\beta_\alpha\in\ivff{-1,1}$, and $H\in\ivoo{0,1}$ is the Hurst exponent. Several regularity properties have been determined in the literature. In particular, sample paths are known to be nowhere bounded \cite{Maejima-1983} if $H < 1/\alpha$ and H\"older continuous when $H > 1/\alpha$. In this latter case, \citet{Takashima-1989,Kono.Maejima-1991} proved that the pointwise and local H\"older exponents satisfy almost surely $H-1/\alpha \leq \alpha_{X,t} \leq H$ and $\widetilde{\alpha}_{X,t} = H-1/\alpha$. Throughout this paper, we will assume that $\alpha\in\ivfo{1,2}$, which is required to obtain H\"older continuous sample paths ($H > 1/\alpha$).

Using an alternative representation of LFSM presented in Proposition~\ref{prop:rep_lfsm}, we enhance the aforementioned regularity results and obtain a precise description of the multifractal structure of the LFSM.
\begin{theorem}  \label{th:2ml_lfsm}
  Suppose $X$ is a linear fractional stable motion parametrized by $\alpha\in\ivfo{1,2}$, $\beta_\alpha\in\ivff{-1,1}$ and $H\in\ivoo{0,1}$. Then, with probability one and for all $\sigma\leq H-\tfrac{1}{\alpha}$
  \begin{equation} \label{eq:2ml_spectrum_lfsm}
    \forall V\in\Oi;\quad \dimH\pth{ E_{\sigma,s'}\cap V } =
    \begin{cases}
      \alpha (s - H) + 1   & \text{ if } s\in\ivffb{H-\tfrac{1}{\alpha},H}; \\
      -\infty   & \text{ otherwise.}
   \end{cases}
  \end{equation}
  where $s = \sigma - s'$. When $\sigma > H-\tfrac{1}{\alpha}$, $E_{\sigma,s'}$ is empty for all $s'\in\R$. 

  \noindent In addition, the weak scaling exponent satisfies with probability one
  \begin{equation} \label{eq:2ml_spectrum_lfsm_weak}
    \forall V\in\Oi; \quad \dimH( E^w_h \cap V) = 
    \begin{cases}
      \alpha (h - H) + 1   & \text{ if } h\in\ivffb{H-\tfrac{1}{\alpha},H}; \\
      -\infty   & \text{ otherwise.}
    \end{cases}
  \end{equation}
  Finally, the chirp scaling exponent $\beta^c_{X,t}$ is equal to $0$ for all $t\in\R$.
\end{theorem}
Therefore, we observe that the multifractal structure presented in Theorem~\ref{th:2ml_lfsm} corresponds to the spectrum of alpha-stable processes translated by a factor $H-\tfrac{1}{\alpha}$.
Interestingly, we also note that on the contrary to usual Hölder exponents, the weak scaling exponent and the 2-microlocal formalism allow to describe the multifractal nature of the LFSM independently of the continuity of its sample paths, unifying the continuous ($H>\tfrac{1}{\alpha}$) and unbounded ($H<\tfrac{1}{\alpha}$) cases (see Figure~\ref{fig:2ml_lfsms}). In the latter case, the 2-microlocal domain is located strictly below the $s'$-axis, implying that sample paths are nowhere bounded. Nevertheless, the proof of Theorem \ref{th:2ml_lfsm} ensures in this case the existence of a modification of the LFSM such that the sample paths are distributions in $\Di'(\R)$ whose 2-microlocal regularity can be studied as well.

In addition, the classic multifractal spectrum can be explicated when sample paths are Hölder continuous.
\begin{corollary}  \label{cor:spectrum_lfsm}
  Suppose $X$ is a linear fractional stable motion parametrized by $\alpha\in\ivfo{1,2}$, $\beta_\alpha\in\ivff{-1,1}$ and $H\in\ivoo{0,1}$, with $H > 1/\alpha$. Then, with probability one, the multifractal spectrum of $X$ is given by
  \begin{equation} \label{eq:spectrum_lfsm}
    \forall V\in\Oi;\quad d_X(h,V) = 
    \begin{cases}
      \alpha \pth{ h - H } + 1   & \text{ if } h\in\ivffb{H - \tfrac{1}{\alpha}, H}; \\
      -\infty   & \text{ otherwise}.
    \end{cases}
  \end{equation}
\end{corollary}
An equivalent multifractal structure is presented in Proposition~\ref{prop:2ml_flp} for a similar class of processes called \emph{fractional L\'evy processes} (see \cite{Benassi.Cohen.ea-2004,Marquardt-2006,Cohen.Lacaux.ea-2008}).\vsp

\begin{figure}[ht!]
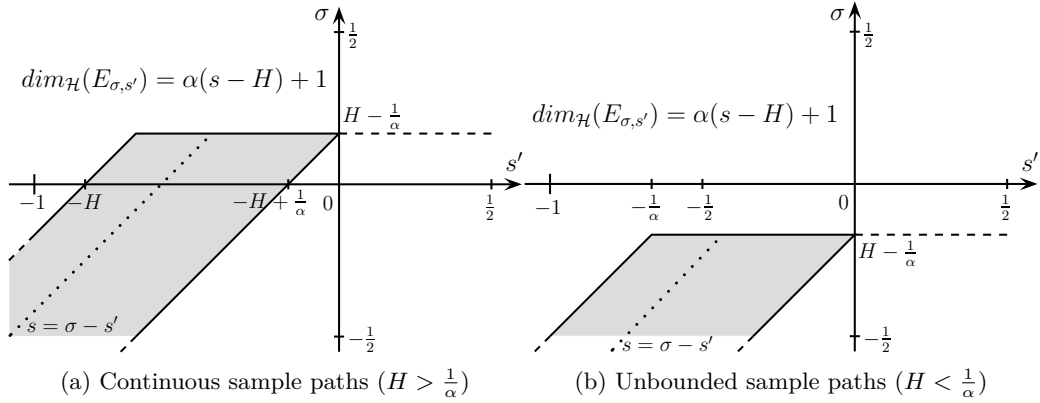

  \centering
  \begin{subfigure}[b]{0.49\textwidth}
    \centering
    \includegraphics[width=\textwidth]{2ml_lfsm1}
    \caption{Continuous sample paths ($H>\tfrac{1}{\alpha}$)}
    \label{fig:2ml_lfsm1}
  \end{subfigure}%
  \begin{subfigure}[b]{0.49\textwidth}
    \centering
    \includegraphics[width=\textwidth]{2ml_lfsm2}
    \caption{Unbounded sample paths ($H<\tfrac{1}{\alpha}$)}
    \label{fig:2ml_lfsm2}
  \end{subfigure}
  \caption{Domains of admissible 2-microlocal frontiers for the LFSM}
  \label{fig:2ml_lfsms}
\end{figure}

The LFSM admits a natural multifractional extension which has been introduced and studied in \cite{Stoev.Taqqu-2004,Stoev.Taqqu-2005,Dozzi.Shevchenko-2011}. The definition of the linear multifractional stable motion (LMSM) is based on Equation~\eqref{eq:def_lfsm}, where the Hurst exponent $H$ is replaced by a function $t\mapsto H(t)$.
\citet{Stoev.Taqqu-2004} and \citet{Ayache.Hamonier-2013} have obtained lower and upper bounds on H\"older exponents which are similar to LFSM results: for all $t\in\R_+$, $H(t)-1/\alpha \leq \alpha_{X,t} \leq H(t)$ and $\widetilde{\alpha}_{X,t} = H(t)-1/\alpha$ almost surely. \citet{Ayache.Hamonier-2013} have also investigated the existence of an optimal local modulus of continuity.

Theorem~\ref{th:2ml_lfsm} can be generalized to the LMSM in the continuous case. More precisely, we assume that the Hurst function satisfies the following assumption,
\begin{equation}  
  H:\R\rightarrow\ivoob{\tfrac{1}{\alpha},1} \text{ is $\delta$-H\"olderian, with }\delta > \sup_{u\in\R} H(u).
\tag{$\Hi_0$}
\end{equation}
Since the LMSM is clearly a non-homogeneous process, it is natural to focus on the study of the spectrum of singularities localized at $t\in\R_+$, i.e. 
\[
   \forall t\in\R_+ \quad d_X(h,t) = \lim_{\rho\rightarrow 0}  d_X(h,B(t,\rho)) = \lim_{\rho\rightarrow 0} \dimH\pth{ E_h \cap B(t,\rho) }.
\]
\begin{theorem}  \label{th:2ml_lmsm}
  Suppose $X$ is a linear multifractional stable motion parametrized by $\alpha\in\ivoo{1,2}$, $\beta_\alpha\in\ivff{-1,1}$ and an $(\Hi_0)$-Hurst function $H$. 

  Then, with probability one, for all $t\in\R$ and for all $\sigma< H(t)-\tfrac{1}{\alpha}$,
  \begin{equation} \label{eq:2ml_spectrum_lmsm}
    \lim_{\rho\rightarrow 0}\ \dimH\pthb{ E_{\sigma,s'} \cap B(t,\rho) } =
    \begin{cases}
      \alpha \pthb{ s - H(t) } + 1   & \text{ if } s\in\ivffb{H(t)-\tfrac{1}{\alpha},H(t)}; \\
      -\infty   & \text{ otherwise.}
   \end{cases}
  \end{equation}
  where $s = \sigma - s'$. Furthermore, the set $E_{\sigma,s'}\cap B(t,\rho)$ is empty for any $\sigma > H(t)-\tfrac{1}{\alpha}$ and $\rho>0$ sufficiently small. 
\end{theorem}
Theorem~\ref{th:2ml_lmsm} extends the results presented in \cite{Stoev.Taqqu-2004,Stoev.Taqqu-2005}, and also ensures that the localized multifractal spectrum is equal to
\begin{equation} \label{eq:spectrum_lmsm}
  \forall t\in\R_+;\quad d_X(h,t)=
  \begin{cases}
    \alpha \pthb{ h - H(t) } + 1   & \text{ if } h\in\ivffb{H(t) - \tfrac{1}{\alpha}, H(t)}; \\
    -\infty   & \text{ otherwise}.
  \end{cases}
\end{equation}
Moreover, we observe that Proposition~\ref{prop:rep_lfsm} and Theorem~ \ref{th:2ml_lmsm} still hold when the Hurst function $H(\cdot)$ is a continuous random process. Thereby, similarly to the works of \citet{Barral.Fournier.ea-2010} and \citet{Durand-2008}, it provides a class stochastic processes whose spectrum of singularities, given by Equation~\eqref{eq:spectrum_lmsm}, is non-homogeneous and random.


\section{L\'evy processes}  \label{sec:2ml_levy}

In this section, $X$ will designate a L\'evy process parametrized by the generating triplet $(a,Q,\pi)$. The L\'evy-It\={o} decomposition states that it can represented as the sum of three independent processes $B$, $N$ and $Y$, where $B$ is a $d$-dimensional Brownian motion, $N$ is a compound Poisson process with drift and $Y$ is a L\'evy process characterized by $\pthb{0,0,\pi(\dt x) \indi_{\brc{ \norm{x}\leq 1 }}}$.

Without any loss of generality, we restrict the study to the time interval $\ivff{0,1}$. Furthermore, as outlined in the introduction, we also assume that the Blumenthal--Getoor $\beta$ is strictly positive. As noted by \citet{Jaffard-1999}, the component $N$ does not affect the regularity of $X$ since its trajectories are piecewise linear with a finite number of jumps. Sample path properties of Brownian motion are well-known and therefore, we first focus in this section on the study of the jump process~$Y$. 

It is well-known that the process $Y$ can be represented as a compensated integral with respect to a Poisson measure $J(\dt t,\dt x)$ of intensity $\mathcal{L}^1\otimes\pi$:
\begin{equation}
  Y_t = \lim_{\eps\rightarrow 0} \bktbb{ \int_{ \ivff{0,t}\times D(\eps,1) } x \,J(\dt s,\dt x) - t\int_{D(\eps,1)} x \,\pi(\dt x) },
\end{equation}
where for all $0\leq a<b$, $D(a,b)\eqdef\brc{x\in\R^d : a<\norm{x}\leq b}$. Moreover, as presented in \cite[Th. 19.2]{Sato-1999}, the convergence is almost surely uniform on any bounded interval.
In the rest of this section, for any $m\in\R_+$, $Y^m$ will denote the L\'evy process:
\begin{equation} \label{eq:notation_levy}
  Y^m_t = \lim_{\eps\rightarrow 0} \bktbb{ \int_{ \ivff{0,t}\times D(\eps,2^{-m}) } x \,J(\dt s,\dt x) - t\int_{D(\eps,2^{-m})} x \,\pi(\dt x) }.
\end{equation}

Finally, in the following proofs, $c$ and $C$ will denote positive constants which can change from a line to another. More specific constants will be written $c_1$, $c_2$, \ldots Furthermore, we will write $u_n\asymp v_n$ when there exists two constants $c_1$, $c_2$ independent of $n$ such that $c_1\, v_n\leq u_n\leq c_2\, v_n$ for every $n\in\N$.

\subsection{Pointwise exponent} \label{ssec:exp_levy}

We extend in this section the multifractal spectrum \eqref{eq:spectrum_levy} to any L\'evy process.
To begin with, we prove two technical lemmas that will be extensively used in the rest of the article.
\begin{lemma} \label{lemma:tech_lemma1}
  For any $\delta > \beta$, there exists a positive constant $c(\delta)$ such that for all $m\in\R_+$
  \[
    \prbb{ \sup_{t\leq 2^{-m}} \normb{ Y^{m/\delta}_t }_1 \geq m 2^{-m/\delta} } \leq c(\delta) e^{-m}.
  \]
\end{lemma}
\begin{proof}
  Let $\delta>\beta$. We observe that for any $m\in\R_+$,
  \[
    \brcbb{ \sup_{t\leq 2^{-m}} \normb{ Y^{m/\delta}_t }_1 \geq m 2^{-m/\delta} } = \bigcup_{\eps\in\brc{-1,1}^d} \brcbb{ \sup_{t\leq 2^{-m}} \scprb{\eps}{ Y^{m/\delta}_t } \geq m 2^{-m/\delta} }
  \]
  Hence, it is sufficient to prove that there exists $c(\delta)>0$ such that for any $\eps\in\brc{-1,1}^d$,
  \[
    \prbb{ \sup_{t\leq 2^{-m}} \scprb{\eps}{ Y^{m/\delta}_t } \geq m 2^{-m/\delta} } \leq  c(\delta) e^{-m}.
  \]
  Let $\lambda = 2^{m/\delta}$ and $M_t = e^{ \lambda\scpr{\eps}{Y^{m/\delta}_t} }$ for all $t\in\R_+$. According to Theorem $25.17$ in \cite{Sato-1999}, we have $\esp{M_t} = \exp\brcb{ t \int_{D(0,2^{-m/\delta})} \pthb{ e^{\lambda\scpr{\eps}{x}} -1 - \lambda\scpr{\eps}{x} } \pi(\dt x) }$. 
  Furthermore, we observe that for all $s\leq t\in\R_+$,
  \[
    \espc{M_t}{\Fi_s} = M_s \exp\brcbb{ (t-s) \int_{D(0,2^{-m/\delta})} \pthb{ e^{\lambda\scpr{\eps}{x}} -1 - \lambda\scpr{\eps}{x} } \pi(\dt x) }  \geq M_s,
  \]
  since for any $y\in\R$, $e^{y} -1 - y \geq 0$.
  Hence, $M$ is a positive submartingale, and using Doob's inequality (Theorem 1.7 in \cite{Revuz.Yor-1999}), we obtain
  \begin{align*}
    \prbb{ \sup_{t\leq 2^{-m}} \scprb{\eps}{ Y^{m/\delta}_t } \geq m 2^{-m/\delta} }
    &= \prbb{ \sup_{t\leq 2^{-m}} M_t \geq e^{m} } 
    \leq e^{-m} \esp{ M_{2^{-m}} }.
  \end{align*}
  For all $y\in\ivff{-1,1}$, we note that $e^{y} -1 - y \leq y^2$. Thus, for any $m\in\R_+$,
  \begin{align*}
    \esp{ M_{2^{-m}} } 
    &\leq \exp\brcbb{ 2^{-m} \int_{D(0,2^{-m/\delta})} \lambda^2 \scpr{\eps}{x}^2 \pi(\dt x) } 
    \leq \exp\brcbb{ 2^{-m} \int_{D(0,2^{-m/\delta})} \lambda^2 \norm{x}^2 \pi(\dt x) }.
  \end{align*}
  If $\beta < 2$, let us set $\gamma>0$ such that $\beta < \gamma < 2$ and $\gamma < \delta$. Then, 
  \begin{align*}
    2^{-m} \int_{D(0,2^{-m/\delta})} \lambda^2 \norm{x}^2 \pi(\dt x) 
    &= 2^{-m(1-2/\delta)} \int_{D(0,2^{-m/\delta})} \norm{x}^{\gamma} \cdot \norm{x}^{2-\gamma} \pi(\dt x) \\
    &\leq 2^{-m(1-2/\delta )} 2^{-m/\delta(2-\gamma)} \int_{D(0,1)} \norm{x}^{\gamma} \pi(\dt x) \\
    &= 2^{-m(1 - \gamma/\delta)} \int_{D(0,1)} \norm{x}^{\gamma} \pi(\dt x) 
    \leq \int_{D(0,1)} \norm{x}^{\gamma} \pi(\dt x),
  \end{align*}
  since $\gamma < \delta$. If $\beta=2$, we simply observe that 
  \begin{align*}
    2^{-m} \int_{D(0,2^{-m/\delta})} \lambda^2 \norm{x}^2 \pi(\dt x) 
    \leq 2^{-m(1-2/\delta)} \int_{D(0,1)} \norm{x}^2 \pi(\dt x) \leq \int_{D(0,1)} \norm{x}^2 \pi(\dt x),
  \end{align*}
  as $\delta > 2$. Therefore, there exists $c(\delta) > 0$ such that for all $m\in\R_+$, $\esp{ M_{2^{-m}} } \leq c(\delta)$, concluding the proof of this lemma.
\end{proof}

\begin{lemma} \label{lemma:tech_lemma2}
  Suppose $\delta > \beta$. Then, with probability one, there exist $c_1>0$ and $M(\omega)>0$ such that
  \begin{equation}  \label{eq:tech_lemma2}
    \forall u,v\in\ivff{0,1} : \abs{u-v}\leq 2^{-m};\quad \normb{ Y^{m/\delta}_u - Y^{m/\delta}_v }\leq c_1 \,m 2^{-m/\delta}
  \end{equation}
  for any $m\geq M(\omega)$.
\end{lemma}
\begin{proof}
  We first note that for any $m\in\R_+$ and any $\delta>\beta$,
  \begin{align*}
    &\brcbb{ \sup_{u,v\in\ivff{0,1} : \abs{u-v}\leq 2^{-m}} \normb{ Y^{m/\delta}_u - Y^{m/\delta}_v } \geq 3m 2^{-m/\delta} } \\
    &\subseteq \bigcup_{k=0}^{2^m-1} \brcbb{ \sup_{t\leq 2^{-m}} \normb{ Y^{m/\delta}_{t+k2^{-m}} - Y^{m/\delta}_{k2^{-m}} } \geq m 2^{-m/\delta} }.
  \end{align*}
  Therefore, the stationarity of L\'evy processes and Lemma~\ref{lemma:tech_lemma1} yield
  \[
    \prbb{ \sup_{u,v\in\ivff{0,1} : \abs{u-v}\leq 2^{-m}} \normb{ Y^{m/\delta}_u - Y^{m/\delta}_v } \geq 3m 2^{-m/\delta} } \leq 2^m c(\delta) e^{-m} = c(\delta) e^{-cm}.
  \]
  Using the latter estimate and Borel--Cantelli lemma, we obtain Equation~\eqref{eq:tech_lemma2}.
\end{proof}

Let us recall the definition of the collection of random sets $(A_\delta)_{\delta > 0}$ introduced by \citet{Jaffard-1999}. For every $\omega\in\Omega$, $S(\omega)$ denotes the countable set of jumps of $Y_\sbullet(\omega)$. Moreover, for any $\eps>0$, let $A^\eps_\delta$ be
\[
  A^\eps_\delta = \bigcup_{ \stackrel{t\in S(\omega)}{\norm{\Delta Y_t}\leq\eps} } \ivffb{ t - \norm{\Delta Y_t}^\delta, t + \norm{\Delta Y_t}^\delta }.
\]
Then, the random set $A_\delta$ is defined by $A_\delta = \limsup_{\eps\rightarrow 0^+} A^\eps_\delta$. As noted in \cite{Jaffard-1999}, if $t\in A_\delta$, we necessarily have $\alpha_{Y,t} \leq \tfrac{1}{\delta}$. The other side inequality is obtained in the next statement which extends Proposition 2 from \cite{Jaffard-1999}.
\begin{proposition} \label{prop:pointwise_levy}
  Suppose $\delta > \beta$. Then, with probability one, for all $t\in\ivff{0,1}\setminus S(\omega)$:
  \[
    t\notin A_\delta \quad\Longrightarrow\quad \alpha_{Y,t} \geq \tfrac{1}{\delta}.
  \]
\end{proposition}
\begin{proof}
  Suppose $\omega\in\Omega$, $t\notin A_\delta$, $u\in\ivff{0,1}$ and $m\in\N$ such that $2^{-(m+1)\delta} \leq \abs{t-u} < 2^{-m\delta}$. Since $t\notin A_\delta$, there exists $\eps_0>0$ such that for all $\eps\leq \eps_0$, $t\notin A^\eps_\delta$. The component
  \[
    \int_{ \ivff{t,u}\times D(\eps_0,1) } x \,J(\dt s,\dt x) - (u-t)\int_{D(\eps_0,1)} x \,\pi(\dt x)
  \]
  is piecewise linear, and therefore does not influence the pointwise exponent $\alpha_{Y,t}$.
  Without any loss of generality, we may assume that $2^{-m}\leq\eps_0$. Then, for any jump $\Delta Y_s$ such that $\norm{\Delta Y_s}\in\ivff{2^{-m},\eps_0}$, we have $\norm{\Delta Y_s}^\delta \geq 2^{-m\delta}\geq\abs{t-u}$, implying that
  \[
    \int_{ \ivff{t,u}\times D(2^{-m},\eps_0) } x \,J(\dt s,\dt x) = 0.
  \]
  Furthermore, using Lemma~\ref{lemma:tech_lemma2}, we obtain
  \[
    \normb{ Y^{m}_u - Y^{m}_t }\leq c \,m 2^{-m} \leq c \log\pthb{ \abs{t-u}^{-1} } \abs{t-u}^{1/\delta},
  \]
  assuming that $\abs{t-u}$ is sufficiently small. Therefore, the remaining term to estimate corresponds to $-(u-t)\int_{D(2^{-m},\eps_0)} x \,\pi(\dt x)$. To study the latter, we distinguish two different cases, depending on the polynomial component we subtract in Definition~\ref{def:pointwise_holder}.
  \begin{enumerate}[ \it 1.]
    \item If $\delta\geq 1$, let us set $P_t \equiv 0$. Then,
    \begin{align*}
      \normbb{ (u-t)\int_{D(2^{-m},\eps_0)} x \,\pi(\dt x) }
      &\leq c\,\abs{t-u} \int_{D(2^{-m},\eps_0)} \norm{x}^\delta \cdot \norm{x}^{1-\delta} \,\pi(\dt x) \\
      &\leq c\,\abs{t-u}\cdot 2^{-m(1-\delta)} \int_{D(2^{-m},\eps_0)} \norm{x}^\delta \,\pi(\dt x) 
      \leq c\,\abs{t-u}^{1/\delta}.
    \end{align*}
    
    \item If $\delta < 1$ (and thus $\beta < 1$), we set $P_t(u) \equiv -(u-t)\int_{D(0,\eps_0)} \,\pi(\dt x)$, which corresponds to the linear drift of the Lévy process. We observe that $-(u-t)\int_{D(2^{-m},\eps_0)} x \,\pi(\dt x) - P_t(u) = (u-t)\int_{D(0,2^{-m})} x \,\pi(\dt x)$. Then, similarly to the previous case, the latter satisfies
    \begin{align*}
      \normbb{ (u-t)\int_{D(0,2^{-m})} x \,\pi(\dt x) }
      &\leq c\,\abs{t-u} \int_{D(0,2^{-m})} \norm{x}^\delta \cdot \norm{x}^{1-\delta} \,\pi(\dt x) \\
      &\leq c\,\abs{t-u}\cdot 2^{-m(1-\delta)} \int_{D(0,2^{-m})} \norm{x}^\delta \,\pi(\dt x) 
      \leq c\,\abs{t-u}^{1/\delta}.
    \end{align*}
  \end{enumerate}
  Therefore, owing to the previous estimates, we have proved that $\norm{Y_u - Y_t - P_t(u)} \leq c_0 \log\pthb{ \abs{t-u}^{-1} } \abs{t-u}^{1/\delta}$, where the constant $c_0$ is independent of $u$. The latter inequality and Definition~\ref{def:pointwise_holder} prove that $\alpha_{Y,t} \geq \tfrac{1}{\delta}$.
\end{proof}

Proposition \ref{prop:pointwise_levy} ensures that almost surely
\begin{equation} \label{eq:charac_Eh}
  \forall h>0; \quad E_h = \pthbb{ \bigcap_{\delta < 1/h} A_\delta } \setminus \pthbb{ \bigcup_{\delta > 1/h} A_\delta } \setminus S\quad\text{ and }\quad E_0 = \pthbb{ \bigcap_{\delta > 0} A_\delta } \cup S.
\end{equation}
Furthermore, since the estimate of the Hausdorff dimension obtained in \cite{Jaffard-1999} does not rely on Assumption \eqref{eq:hyp_jaffard}, the Lévy process $Y$ satisfies with probability one
\[
  \forall V\in\Oi;\quad \dimH (E_h\cap V) = 
  \begin{cases}
    \beta h \quad &\text{if}\quad h\in\ivff{0,1/\beta}; \\
    -\infty       &\text{otherwise.}
  \end{cases}
\]

\subsection{2-microlocal frontier of Lévy processes} \label{ssec:2ml_levy}

We now aim to refine the multifractal spectrum of Lévy processes by studying their 2-microlocal structure. Let us begin with a few basics remarks and estimates on their 2-microlocal frontier. Firstly, according to \cite[Th. 3.13]{Meyer-1998}, with probability one, for all $t\in\ivff{0,1}$ and for any $-s' < \alpha_{Y,t}$, the sample path $Y_\sbullet(\omega)$ belongs to the 2-microlocal space $C^{0,s'}_t$. Furthermore, owing to the density of the set of jumps $S(\omega)$ in $\ivff{0,1}$, necessarily $Y_\sbullet(\omega) \notin C^{\sigma,s'}_t$ for any $\sigma > 0$ and all $s'\in\R$.
Hence, since the 2-microlocal frontier is a concave function with left- and right-derivatives between $0$ and $1$, with probability one and for all $t\in\ivff{0,1}$:
\[
  \forall s'\in\R_+;\quad \sigma_{Y,t}(s') \geq (\alpha_{Y,t} + s')\wedge 0\quad\text{and}\quad \sigma_{Y,t}(s') \leq 0.
\]
Therefore, we are interested in obtaining finer estimates of the negative component of the 2-microlocal frontier of $Y$. As outlined in the introduction and Definitions~\ref{def:2ml_spaces}-\ref{def:2ml_spaces_gen}, we need to analyse the following type of increments in the neighbourhood of $t$:
\begin{equation}  \label{eq:levy_int_incr}
  \normbb{ \int_b^u (u-s)_+^{k-1} Y_s \,\dt s -P_{t,k}(u) - \int_b^v (v-s)_+^{k-1} Y_s \,\dt s + P_{t,k}(v) }
\end{equation}
where $b < t$ is fixed and $k\geq 1$. The polynomial component to be subtracted can be estimate using our work on pointwise exponent. Indeed, when $k=0$, the $P_{t,0}\equiv P_t$ where the latter has been presented in the proof of Proposition~\ref{prop:pointwise_levy},. Then, the consistency of the definition of the 2-microlocal spaces imposes that $P_{t,k-1}$ must correspond to the derivative of $P_{t,k}$. This last property shows us that the form of $P_{t,k}$ can be completely deduce from the known polynomial $P_t$.

For the sake of readability, we divide the proof of Theorem~\ref{th:2ml_levy} and its corollaries in several different technical lemmas. To begin with, we give simple estimates on the jumps of a Lévy process.
\begin{lemma} \label{lemma:levy_jumps1}
  For any $\eps>0$, there exists an increasing sequence $(m_n)_{n\in\N}$ such that with probability one, for all $t\in\ivff{0,1}$ and for every $n\geq N(\omega)$
  \[
    \exists u\in B(t,2^{-m_n\alpha});\quad \norm{\Delta Y_u} \geq 2^{-m_n} \quad\text{and}\quad J\pthb{ B(u,2^{-m_n \gamma})\times D(2^{-m_n(1+\eps)},1) } = 1,
  \]
  where $\alpha = \beta(1-2\eps)$ and $\gamma = \beta(1+4\eps)$.
\end{lemma}
\begin{proof}
  Suppose $m\in\N$, $\eps>0$, $\alpha = \beta(1-2\eps)$ and $\gamma = \beta(1+4\eps)$. Let $I$ be an interval such that $I=I_1\cup I_2\cup I_3$, where $I_1,I_2,I_3$ are three consecutive and disjoint intervals of size $2^{-m_n \gamma}$. Then, we are interested in the following event:
  \begin{align*}
    A = 
    &\brcb{ J\pthb{ I_1, D(2^{-m},1) }=0 } \cap \brcb{ J\pthb{ I_3, D(2^{-m},1) }=0 } \,\cap \\
    &\brcb{ J\pthb{ I_2, D(2^{-m},1) }=1 }\cap\brcb{ J\pthb{ I,D(2^{-m(1+\eps)},2^{-m}) }=0 },
  \end{align*}
  Since $J$ is a Poisson measure, $A$ corresponds to the intersection of independent events whose probability is equal to 
  \[
    \pr{A} = 2^{-m\gamma} \pi(D(2^{-m},1))\cdot \exp\pthb{ -3\cdot 2^{-m\gamma} \pi(D(2^{-m},1)) + 2^{-m\gamma} \pi(D(2^{-m(1+\eps)},2^{-m})) }.
  \]
  As described in \cite{Blumenthal.Getoor-1961}, $\beta$ can be defined by $\beta = \inf\brcb{ \delta \geq 0 : \limsup_{r\rightarrow 0} r^{\delta} \pi\pthb{B(r,1)} < \infty }$.
  Therefore, there exists $r_0 > 0$ such that for all $r\in\ivof{0,r_0}$, $\pi\pthb{B(r,1)} \leq r^{-\beta(1+\eps)}$. Hence, for any $m\in\N$ sufficiently large:
  \[
    \pr{A} \geq 2^{-m\gamma} \pi(D(2^{-m},1)) \, \exp\pthb{ -2^{-m\beta\eps+1} } \geq 2^{-m\gamma-1} \,\pi\pthb{ D(2^{-m},1) }.
  \]
  Furthermore, according to the definition of $\beta$, there also exists an increasing sequence $(m_n)_{n\in\N}$ such that for all $n\in\N$, $\pi(D(2^{-m_n},1))\geq 2^{m_n \beta(1-\eps)}$. Therefore, along this sequence, we obtain $\pr{A} \geq 2^{-m_n 5\beta\eps-1}$ for every $n\in\N$.

  Let now consider an interval $\Ii$ of size $2^{-m_n\alpha}$. There exist at most $2^{-m_n\alpha+m_n\gamma}$ disjoint sub-intervals $I$ of size $3\cdot2^{-m_n \gamma}$. We designate by $B$ the event where $A$ is not satisfied by all these sub-elements $I$. Owing to the previous estimate of $\pr{A}$ and the independence of these different events, we obtain
  \[
    \pr{B} = \pthb{ P(A^c) }^{ 2^{m_n(\gamma-\alpha)} } \leq \pthb{ 1 - 2^{-m_n 5\beta\eps-1} }^{ 2^{m_n(\gamma-\alpha)} }.
  \]
  Note that $\gamma-\alpha = 6\beta\eps$. Hence, $\log\pthb{\pr{B}} \leq - 2^{-m_n 5\beta\eps-1}\cdot 2^{m_n 6\beta\eps} = -2^{m_n\beta\eps-1}$ and the probability $\pr{B}$ satisfies $\pr{B} \leq \exp\pthb{ -2^{m_n\beta\eps-1} }$.

  Finally, we know there exist at most $2^{m_n\alpha+1}$ disjoint intervals $J$ of size $2^{-m_n\alpha}$ inside $\ivff{0,1}$. We denote by $B_n$ the event where $B$ is satisfied for one of the previous interval $\Ii$. Since $B_n$ is the reunion of events, we obtain 
  \[
    \pr{B_n} \leq 2^{m_n\alpha+1}\cdot\exp\pthb{ -2^{m_n\beta\eps-1} } \leq \exp\pthb{ -2^{m_n\beta\eps-1} + c\, m_n\alpha }.
  \]
  Therefore, $\sum_{n\in\N} \pr{B_n} < \infty$ and owing to Borel--Cantelli lemma, with probability one, there exists $N(\omega)$ such that for every $n\geq N(\omega)$, $\omega\in B^c_n$. The latter inclusion means that for every interval $\Ii$ previous defined, there exists a sub-element $I$ such that the event $A$ is satisfied on $I$, which proves this lemma.
\end{proof}

The previous lemma will help us to obtain a uniform upper bound on the 2-microlocal frontier of $Y$.
\begin{lemma} \label{lemma:2ml_levy1}
  With probability one, for all $t\in\ivff{0,1}$, the 2-microlocal frontier of $Y$ at $t$ satisfies
  \begin{equation}  \label{eq:lemma1_ineq_2ml}
    \forall s'\in\R;\quad \sigma_{Y,t}(s') \leq \pthB{ \frac{1}{\beta} + s' }\wedge 0.
  \end{equation}
\end{lemma}
\begin{proof}
  Let us first observe that to obtain an upper bound of the 2-microlocal frontier of the $\R^d$-valued $Y=(Y_1,\dotsc,Y_d)$ process, it is sufficient to prove this bound holds for one the component $Y_i$. Furthermore, we also know that each of these components is a one-dimensional Lévy process and there exists $i\in\brc{1,\dotsc,d}$ such that the Blumenthal--Getoor exponent of $Y_i$ is equal to $\beta$. Hence, considering these two remarks, we may assume without any loss of generality that we study only one component, and thus $d=1$.

  Let us set $t\in\ivff{0,1}$. We need to evaluate the size of the increments described in Equation~\eqref{eq:levy_int_incr}. Hence, let us first determine the form of the local process $Y(u,k)\eqdef\pthb{ I^k_{b+}Y }(u)-P_{t,k}(u)$ used. We know that when $k=0$, the polynomial component is described in Proposition~\ref{prop:pointwise_levy}, and thus we define the local process $Y(s,0)$ in the neighbourhood of $t$ as following:
  \[
    \forall u\in\R;\quad Y(u,0) = Y_u - Y_t - P_t(u).
  \]
  Then, since the polynomial component must correspond to the Taylor development of the process at $t$, we define the elements $Y(\cdot,k)$ be induction:
  \[
    \forall u\in\R;\quad Y(u,k) = \int_t^u Y(s,k-1) \,\dt s.
  \]
  One can easily verify that the derivative of $Y(\cdot,k)$ is $Y(\cdot,k-1)$ and $Y(t,k)=0$, proving that the Taylor development of $Y(\cdot,k)$ at $t$ is $P\equiv 0$. Therefore, this construction procedure ensures that the difference between $Y(\cdot,k)$ and $I^k_{b+}Y$ corresponds to the polynomial function appearing in the definition of the 2-microlocal spaces.

  Hence, we need to show in this proof that for any $k\in\N$, the increments of the process $Y(\cdot,k)$ are sufficiently large in the neighbourhood of $t$. More precisely, we will show by induction that there exist $t_{n,k}\rightarrow_n t$, $\rho_{n,k}>0$ and $\delta_{n,k}>0$ such that for every $k\in\N$ and all $n\in\N$:
  \begin{equation}  \label{eq:2ml_induction}
    \forall u\in\ivfo{t_{n,k},t_{n,k}+\rho_{n,k}};\quad \abs{ Y(u,k) } \geq \delta_{n,k}.
  \end{equation}
  To initialize the induction with $k=0$, we make use of the estimate obtained in Lemma~\ref{lemma:levy_jumps1}: there exists an increasing sequence $(m_n)_{n\in\N}$ such that with probability one, for all $t\in\ivff{0,1}$ and for every $n\geq N(\omega)$
  \[
    \exists v\in B(t,2^{-m_n\alpha});\quad \abs{\Delta Y_v} \geq 2^{-m_n} \quad\text{and}\quad J\pthb{ B(v,2^{-m_n \gamma})\times D(2^{-m_n(1+\eps)},1) } = 1,
  \]
  where $\alpha = \beta(1-2\eps)$ and $\gamma = \beta(1+4\eps)$. Since the reasoning which follows is completely symmetric, we may assume without any loss of generality that $v\geq t$ and $\Delta Y_v\geq 0$. Let us set $n\geq N(\omega)$ and a proper $v\geq t$. We know there is no other jump of size greater than $2^{-m_n(1+\eps)}$ inside the ball $B(v,2^{-m_n \gamma})$. Therefore, for all $u\in B(v,2^{-m_n \gamma})$, 
  \[
    Y_u-Y_v = -\Delta Y_v\indi_{\brc{u<v}} -(u-v) \int_{D(2^{-m_n(1+\eps)},1)} x \,\pi(\dt x) + Y_u^{m_n(1+\eps)}-Y_v^{m_n(1+\eps)}. 
  \]
  Furthermore, according to Lemma~\ref{lemma:tech_lemma2}, the norm of the latter increment satisfies:
  \[
    \absb{ Y_u^{m_n(1+\eps)}-Y_v^{m_n(1+\eps)} } \leq c_1 \,m_n 2^{-m_n(1+\eps)},
  \]
  as we note that $\abs{u-v}\leq 2^{-m_n\beta(1+4\eps)} = 2^{-m_n(1+\eps)\beta(1+4\eps)/(1+\eps)}$ with $\beta(1+4\eps)/(1+\eps) > \beta$. 
  Then, similarly to the proof of Proposition~\ref{prop:pointwise_levy}, we need to distinguish two different cases.
  \begin{enumerate}[ \it 1.]
    \item If $\beta\geq 1$, $P_t\equiv 0$ and thus $Y(u,0) = Y_u - Y_t$. Let us first assume that $Y(v,0)\geq 2^{-m_n-1}$ and set $t_{n,0} = v$ and $\rho_{n,0} = 2^{-m_n\gamma}$. Then, for all $u\in\ivfo{t_{n,0},t_{n,0}+\rho_{n,0}}$:
    \begin{align*}
      \abs{Y(u,0)} \geq \abs{Y(v,0)} - \absb{ Y_u^{m_n(1+\eps)}-Y_v^{m_n(1+\eps)} } - \absbb{(u-v) \int_{D(2^{-m_n(1+\eps)},1)} x \,\pi(\dt x)}.
    \end{align*}
    Using the estimates presented in Proposition~\ref{prop:pointwise_levy}, we obtain an upper bound of the last term:
    \[
      \absbb{(u-v) \int_{D(2^{-m_n(1+\eps)},1)} x \,\pi(\dt x)}
      \leq c\,2^{-m_n\gamma} \cdot 2^{-m_n(1+\eps)(1-\beta(1+\eps))}
      \leq c\,2^{-m_n(1+\eps)},
    \]
    since $-\gamma + \beta(1+\eps)^2 = -\beta(2\eps-\eps^2) < 0$. Hence, $\abs{Y(u,0)} \geq 2^{-m_n-1} - c\,2^{-m_n(1+\eps)} \geq 2^{-m_n-2}$ for any $n$ sufficiently large. 

    Let now assume that $Y(v,0)\leq 2^{-m_n-1}$. Since $\Delta Y_v\geq 2^{-m_n}$, we necessarily have $Y(v_-,0)\leq -2^{-m_n-1}$. Then, we set in this case $t_{n,0} = v-2^{-m_n\gamma}$ and $\rho_{n,0} = 2^{-m_n\gamma}$, and obtain as well $\abs{Y(u,0)} \geq 2^{-m_n-1} - c\,2^{-m_n(1+\eps)} \geq 2^{-m_n-2}$.\vsp

    \item If $\beta < 1$, $P_t(u) \equiv -(u-t)\int_{D(0,1)} \,\pi(\dt x)$. Similarly to the previous case, we first assume that $Y(v,0)\geq 2^{-m_n-1}$ and set $t_{n,0} = v$ and $\rho_{n,0} = 2^{-m_n\gamma}$. Then, for all $u\in\ivff{t_{n,0},t_{n,0}+\rho_{n,0}}$:
    \begin{align*}
      \abs{Y(u,0)} \geq \abs{Y(v,0)} - \absb{ Y_u^{m_n(1+\eps)}-Y_v^{m_n(1+\eps)} } - \absbb{(u-v) \int_{D(0,2^{-m_n(1+\eps)})} x \,\pi(\dt x)},
    \end{align*}
    where the latter element satisfies
    \[
      \absbb{(u-v) \int_{D(0,2^{-m_n(1+\eps)})} x \,\pi(\dt x)}
      \leq c\,2^{-m_n\gamma} \cdot 2^{-m_n(1+\eps)(1-\beta(1+\eps))}
      \leq c\,2^{-m_n(1+\eps)}.
    \]
    Hence, $\abs{Y(u,0)} \geq 2^{-m_n-1} - c\,2^{-m_n(1+\eps)} \geq 2^{-m_n-2}$ for any $n$ sufficiently large. In the case $Y(v,0)\leq 2^{-m_n-1}$, we observe that $Y(v_-,0)\leq -2^{-m_n-1}$. Therefore, setting $t_{n,0} = v-2^{-m_n\gamma}$ and $\rho_{n,0} = 2^{-m_n\gamma}$, we obtain $\abs{Y(u,0)} \geq 2^{-m_n-1} - c\,2^{-m_n(1+\eps)} \geq 2^{-m_n-2}$.
  \end{enumerate}
  Therefore, in both cases, we have proved that 
  \[
    \forall u\in\ivfo{t_{n,0},t_{n,0}+\rho_{n,0}}; \quad\abs{Y(u,0)} \geq \delta_{n,0}, 
  \]
  where $\rho_{n,0} = 2^{-m_n\gamma}$, $\delta_{n,0}=2^{-m_n-1}$ and $B(t_{n,0},\rho_{n,0})\subset B(t,2^{-m_n\alpha+1})$.\vsp

  Let now assume that Equation~\eqref{eq:2ml_induction} is satisfied for $k\in\N$. Without any loss of generality, we may suppose that $Y(u,k) \geq \delta_{n,k}$ on the interval $\ivfo{t_{n,k},t_{n,k}+\rho_{n,k}}$ (otherwise, simply consider the process $-Y(u,k)$ in the following reasoning). In this case, the function $u\mapsto \int_t^u Y(s,k) \,\dt s$ is strictly increasing on the previous interval. 

  Let us first assume that $\int_t^{t_{n,k}+\rho_{n,k}/2} Y(s,k) \,\dt s \geq 0$. In this case, we set $t_{n,k+1} = t_{n,k}+3/4 \rho_{n,k}$, $\rho_{n,k+1} = \rho_{n,k} / 4$ and $\delta_{n,k+1} = \rho_{n,k}\delta_{n,k} / 4$. Then, for all $u\in\ivfo{t_{n,k+1},t_{n,k+1}+\rho_{n,k+1}}$
  \begin{align*}
    Y(u,k+1)
    \geq \int_{t_{n,k}+\rho_{n,k}/2}^{u} Y(s,k) \,\dt s
    &\geq \pth{u-t_{n,k}-\rho_{n,k}/2}\,\delta_{n,k} \\
    &\geq \rho_{n,k}\delta_{n,k} / 4 = \delta_{n,k+1},
  \end{align*}
  In the other case $\int_t^{t_{n,k}+\rho_{n,k}/2} Y(s,k) \,\dt s \leq 0$, we consider the set of parameters $t_{n,k+1} = t_{n,k}$, $\rho_{n,k+1} = \rho_{n,k} / 4$ and $\delta_{n,k+1} = \rho_{n,k}\delta_{n,k} / 4$. Then, for all $u\in\ivfo{t_{n,k+1},t_{n,k+1}+\rho_{n,k+1}}$
  \begin{align*}
    Y(u,k+1)
    \leq -\int_{u}^{t_{n,k}+\rho_{n,k}/2} Y(s,k) \,\dt s
    &\leq -\pth{t_{n,k}+\rho_{n,k}/2 - u}\,\delta_{n,k} \\
    &\leq -\rho_{n,k}\delta_{n,k} / 4 = -\delta_{n,k+1},
  \end{align*}
  Therefore, assuming that Equation~\eqref{eq:2ml_induction} holds for $k\in\N$, we have proved that it does too for $k+1$ with $\rho_{n,k+1} = \rho_{n,k} / 4$, $\delta_{n,k+1} = \rho_{n,k}\delta_{n,k} / 4$ and $B(t_{n,k+1},\rho_{n,k+1})\subset B(t_{n,k},\rho_{n,k})\subset B(t,2^{-m_n\alpha+1})$.\vsp

  Finally, the lower bound on $Y_{n,k}$ presented in Equation~\eqref{eq:2ml_induction} will now help us to obtain the expected bound on the 2-microlocal frontier. Owing to the previous definitions, for every $k\in\N$, $\abs{t_{n,k}-t} \leq 2^{-m_n\alpha+1}$ and there exist $c_k>0$ independent of $n\in\N$ such that $\delta_{n,k} = c_k\, 2^{-m_n \,k\gamma} \cdot 2^{-m_n}$. Hence, for every $n\in\N$,
  \begin{align*}
    \abs{ Y(t_{n,k},k)-Y(t,k) } 
    \geq \delta_{n,k}
    \geq  c_k\, 2^{-m_n(1+k\gamma)} \geq c_k\, \abs{t_{n,k}-t}^{(1+k\gamma)/\alpha},
  \end{align*}
  where we recall that $\alpha = \beta(1-2\eps)$ and $\gamma = \beta(1+4\eps)$. Therefore, this last inequality proves that the pointwise exponent of $Y(\cdot,k)$ at $t$ satisfies $\alpha_{Y(\cdot,k),t} \leq (1+k\gamma)/\alpha \rightarrow_{\eps\rightarrow 0} 1/\beta + k$. Owing to the Definition~\ref{def:2ml_spaces_gen} of the 2-microlocal spaces, this last inequality induces that with probability one, for any $t\in\ivff{0,1}$ and all $s'\in\R$, $\sigma_{Y,t}(s') \leq \pthb{ \frac{1}{\beta} + s' }\wedge 0$.
\end{proof}

As we have obtained a uniform upper bound on the 2-microlocal frontier, we now study more precisely the regularity of $Y$ at times where $\alpha_{Y,t} < 1 / \beta$. To begin with, we prove a simple lemma related to the number of jumps inside an interval.
\begin{lemma} \label{lemma:tech_lemma3}
  Suppose $\delta > \beta$, $\eps>0$ and $k\in\N$ are such that $\delta > \beta (1+2\eps)(k+1) / k$. Then, with probability one, there exists $M(\omega)$ such that:
  \begin{equation}
    \forall t\in\ivff{0,1};\quad J\pthb{ B(t,2^{-m\delta}), D(2^{-m(1+\eps)},1) } \leq k,
  \end{equation}
  for every $m\geq M(\omega)$.
\end{lemma}
\begin{proof}
  Let $m\in\N$ and $I$ be an interval of size $2^{-m\delta+2}$. Since $J$ is a Poisson random measure,
  \begin{align*}
    \prb{ J\pthb{I,D(2^{-m(1+\eps)},1)} > k } = \exp\pth{ -\lambda_m} \brcbb{ \sum_{\ell = k+1}^{+\infty}  \frac{\lambda_m^\ell}{\ell!} },
  \end{align*}
  where $\lambda_m = 2^{-m\delta+2}\,\pi\pthb{D(2^{-m(1+\eps)},1)}\leq 2^{-m\delta+2 + m\beta(1+2\eps)}\rightarrow 0$. Hence, we obtain the inequality $\prb{ J\pthb{I,D(2^{-m(1+\eps)},1)} > k } \leq c\, \lambda_m^{k+1}$. 

  Considering a covering of the interval $\ivff{0,1}$ with $\ceil{2^{m\delta}}$ overlapping sub-elements $I$ of size $2^{-m\delta+2}$, we denote by $B_m$ the event where at least one of these intervals has more than $k$ jumps inside. Then,
  \[
    \pr{B_m} \leq c\,2^{m\delta}\cdot\lambda_m^{k+1} \leq c\, 2^{m\delta-m\delta(k+1)+m\beta(k+1)(1+2\eps)}.
  \]
  Since $\delta k > \beta (1+2\eps)(k+1)$, there exists $\gamma>0$ such $\pr{B_m} \leq c\,2^{-m\gamma}$. Therefore, $\sum_{m\in\N} \pr{B_m} < \infty$ and owing to Borel--Cantelli lemma, there exists $M(\omega)$ such that for every $m\geq M(\omega)$, $J\pthb{I,D(2^{-m(1+\eps)},1)} \leq k$. Finally, since we consider intervals $I$ of size $2^{-m\delta+2}$ covering $\ivff{0,1}$ and overlapping, we have proved that for all $t\in\ivff{0,1}$, $J\pthb{ B(t,2^{-m\delta}), D(2^{-m(1+\eps)},1) } \leq k$.
\end{proof}

In the next lemma, we start with the study of the 2-microlocal frontier of $Y$ at points $t\in\ivff{0,1}$ where $\alpha_{Y,t}\in\ivff{0,1 / 2\beta}$.
\begin{lemma}  \label{lemma:2ml_levy2}
  With probability one, for all $h\in\ivff{0,1/2\beta}$, the singularities of $Y$ satisfy $\widetilde{E}_h = E_h$ and $\widehat{E}_h=\emptyset$, i.e. for all $t\in E_h$
  \[
    \forall s'\in\R;\quad \sigma_{Y,t}(s') = ( \alpha_{Y,t}+s' )\wedge 0.
  \]
\end{lemma}
\begin{proof}
  Suppose $h\in\ivff{0,1 / 2\beta}$ and $t\in E_h\setminus S(\omega)$ ($t$ is not a jump time). Since we already know that $\sigma_{Y,t}(s')\geq ( h+s' )\wedge 0$, we need to only prove the other side inequality. For that purpose, we will proceed similarly to the proof of Lemma~\ref{lemma:2ml_levy1}.

  More precisely, let us set $\eps>0$ and $\delta > \max\pthb{ 2\beta(1+2\eps), 1/(h+\eps)}$. Since $t\in E_h$ and owing to Equation~\eqref{eq:charac_Eh}, there exist two sequences $(v_n)_{n\in\N}$ and $(m_n)_{n\in\N}$ such that 
  \[
    \forall n\in\N;\quad v_n\in B\pthb{ t,2^{-m_n / (h+\eps)} }\quad\text{ and }\quad \normb{\Delta Y_{v_n}} \geq 2^{-m_n}.
  \]
  Without any loss of generality, we may assume that $v_n\geq t$.
  Furthermore, owing to Lemma~\ref{lemma:tech_lemma3}, $J\pthb{ B(v_n,2^{-m_n\delta}), D(2^{-m_n(1+\eps)},1) } = 1$, i.e. there is no other jump larger than $2^{-m_n(1+\eps)}$ in the neighbourhood of $v_n$.
  Then, similarly to the proof of Lemma~\ref{lemma:2ml_levy1}, we need to distinguish two different cases.
  \begin{enumerate}[ \it 1.]
    \item If $\beta\geq 1$, $P_t\equiv 0$ and thus $Y(u,0) = Y_u - Y_t$. Consider $n\in\N$ and first assume that $\norm{Y(v_n,0)}\geq 2^{-m_n-1}$. Let us also set $t_n = v_n$ and $\rho_n = 2^{-m_n\delta}$. Then, for all $u\in\ivfo{t_{n},t_{n}+\rho_{n}}$:
    \begin{align*}
      \norm{Y(u,0)} \geq \norm{Y(v_n,0)} - \normb{ Y_u^{m_n(1+\eps)}-Y_{v_n}^{m_n(1+\eps)} } - \normbb{(u-v_n) \int_{D(2^{-m_n(1+\eps)},1)} x \,\pi(\dt x)}.
    \end{align*}
    Still using the same estimates as in the proof of Lemma~\ref{lemma:tech_lemma3}, we know that the last two terms are upper bounded by $ c\,2^{-m_n(1+\eps)}$, proving that $\norm{Y(u,0)}\geq 2^{-m_n-2}$ for any $n$ sufficiently large. The case $\norm{Y(v_n,0)}\leq 2^{-m_n-1}$ is treated completely similarly, using $t_n = v_n-2^{-m_n\delta}$ and $\rho_n = 2^{-m_n\delta}$.\vsp

    \item If $\beta < 1$, $P_t(u) \equiv -(u-t)\int_{D(0,1)} \,\pi(\dt x)$. Assuming first  that $\norm{Y(v_n,0)}\geq 2^{-m_n-1}$. we still set $t_{n} = v_n$ and $\rho_{n} = 2^{-m_n\delta}$. Then, for all $u\in\ivfo{t_{n},t_{n}+\rho_{n}}$:
    \begin{align*}
      \norm{Y(u,0)} \geq \norm{Y(v_n,0)} - \normb{ Y_u^{m_n(1+\eps)}-Y_{v_n}^{m_n(1+\eps)} } - \normbb{(u-v_n) \int_{D(0,2^{-m_n(1+\eps)})} x \,\pi(\dt x)},
    \end{align*}
    As previously, the last two terms are upper bounded by $ c\,2^{-m_n(1+\eps)}$, proving that $\norm{Y(u,0)}\geq 2^{-m_n-2}$ for any $n$ sufficiently large. The case $\norm{Y(v_n,0)}\leq 2^{-m_n-1}$ is treated similarly using $t_n = v_n-2^{-m_n\delta}$ and $\rho_n = 2^{-m_n\delta}$.
  \end{enumerate}
  Therefore, we have proved in both cases that for all $u\in\ivfo{t_{n},t_{n}+\rho_{n}}$, with $n$ sufficiently large, $\norm{Y(u,0)}\geq 2^{-m_n-2}$. Reproducing the same reasoning as in the proof of Lemma~\ref{lemma:2ml_levy1}, there exists $s_n$ such that for every $n\in\N$, $s_n\in  B(t,2^{-m_n / (h+\eps)})$ and 
  \begin{align*}
    \norm{ Y(s_n,1) } 
    \geq c\, 2^{-m_n}\cdot 2^{-m_n\delta}
    \geq \abs{t-s_n}^{(h+\eps)(1+\delta)}.
  \end{align*}
  Hence, $\alpha_{Y(\cdot,1),t} \leq (h+\eps)(1+\delta)$. Considering the limit $\eps\rightarrow 0$ and $\delta\rightarrow 1/h$, we obtain $\alpha_{Y(\cdot,1),t} \leq h+1$. The latter inequality is sufficient to prove that $\sigma_{Y,t}(s') = (h+s')\wedge 0$ for all $s'\in\R$.

  To conclude this proof, let us consider the case $t\in S(\omega)$. We observe that for all $u\geq t$,
  \[
    \int_t^u Y_s\,\dt s = (u-t)Y_t + \int_t^u (Y_s - Y_t)\,\dt s\quad\text{with }\ \normbb{ \int_t^u (Y_s - Y_t)\,\dt s } = \text{o}(\abs{t-u}),
  \]
  as $Y$ is right-continuous. Similarly, for all $u\leq t$, $\int_u^t Y_s\,\dt s = (t-u)Y_{t-} + \text{o}(\abs{t-u})$. Therefore, since $\Delta Y_t = Y_t - Y_{t-} \neq 0$, there does not exist a polynomial $P_t$ which can cancel both terms $(u-t)Y_{t}$ and $(t-u)Y_{t-}$, proving that $\sigma_{Y,t}(s') = s'\wedge 0$ for all $s'\in\R$.
\end{proof}

In the last technical lemma, we focus on the particular case $\alpha_{Y,t}\in\ivoo{1 / 2\beta, 1 / \beta}$ and try to distinguish oscillating singularities from the common cusp behaviour.
\begin{lemma}  \label{lemma:2ml_levy3}
  With probability one, for all $h\in\ivoo{1/2\beta,1/\beta}$, $Y$ satisfies
  \begin{equation}  \label{eq:2ml_lemma3}
    \forall V\in\Oi;\quad \dimH( \widetilde{E}_h\cap V ) = \beta h\quad\text{ and }\quad \dimH( \widehat{E}_h\cap V ) \leq 2\beta h - 1 \ ( < \beta h ).
  \end{equation}
  Furthermore, for any $t\in \widehat{E}_h$ and all $s'\in\R$, $\sigma_{Y,t}(s') \leq (h+s') / 2\beta h$.
\end{lemma}
\begin{proof}
  On the contrary to the previous lemma, we know that some oscillating singularities might appear at a given time $t$. Hence, the first step in this proof is to isolate these behaviours and estimate the fractal dimension of the corresponding set of times.

  For that purpose, let us set $\delta\in\ivoo{\beta,2\beta}$ and $\eps>0$. We are interested in the \emph{double-jump} configurations, i.e. when two jumps greater than $2^{-m(1+\eps)}$ are sufficiently close. More precisely, suppose $I$ is an interval of size $2^{-m\delta+1}$ and $p_m$ designates the probability of obtaining at least two jumps greater than $2^{-m(1+\eps)}$ inside $I$. Then,
  \[
    c\, 2^{-m2\delta}\cdot 2^{m2\beta} \leq p_m = \prb{ J\pthb{I,D(2^{-m(1+\eps)},1)} \geq 2 } \leq C\, 2^{-m2\delta}\cdot 2^{m2\beta(1+2\eps)}
  \]
  where we may assume that $\delta > \beta(1+2\eps)$. We consider $\ceil{2^{m\delta}}$ consecutive, but disjoint, intervals of size $2^{-m\delta+1}$ which are sufficient to cover $\ivff{0,1}$. Then, if we denote by $N^1_m$ the number of intervals with the previous configuration, it follows a Binomial distribution of parameters $\ceil{2^{m\delta}}$ and $p_m$. Moreover, Chernoff's inequality induces that
  \[
    \prb{ N^1_m \geq c_0\, 2^{-m\delta+m2\beta(1+2\eps)} } \leq \exp\pthb{-c \,2^{-m\delta+m2\beta} },
  \]
  where $2\beta>\delta$. Let us consider now the same configurations of intervals translated by $2^{-m\delta}$ and denote by $N^2_m$ the corresponding Binomial random variable. Owing to the previous estimate and Borel--Cantelli lemma, with probability one, there exists $M(\omega)$ such that for every $m\geq M(\omega)$, $N^1_m \leq c_0\, 2^{-m\delta+m2\beta(1+2\eps)}$ and $N^2_m \leq c_0\, 2^{-m\delta+m2\beta(1+2\eps)}$. 

  Then, let $T_m$ index the previous intervals with a double-jump configuration and $F(\delta,\eps)$ designate the following set: 
  \[
    F(\delta,\eps) = \limsup_{m\rightarrow\infty} \bigcup_{I\in T_m} \ivffb{c(I)-2^{-m\delta+2},c(I)+2^{-m\delta+2}},
  \]
  where $c(I)$ denotes the center of any interval $I\in T_m$. Using a simple covering based on intervals of size $2^{-m\delta}$, we can obtain an upper bound of the Hausdorff dimension of $F(\delta,\eps)$. More precisely, for any $m_0\in\N$, the series
  \[
    \sum_{m=m_0}^{+\infty} c\,\abs{T_m} \cdot \pthb{2^{-m\delta}}^\gamma \leq c \sum_{m=m_0}^{+\infty} 2^{-m(\delta(1+\gamma)-2\beta(1+2\eps))},
  \]
  converges when $\delta(1+\gamma)>2\beta(1+2\eps)$, i.e. $\gamma > 2\beta(1+2\eps)/\delta - 1$. Therefore, $\dimH\, F(\delta,\eps) \leq 2\beta(1+2\eps)/\delta - 1$ almost surely.

  Let now set $h\in\ivoo{1/2\beta,1/\beta}$. We aim to prove that $\widehat{E}_h \subset F(\delta,\eps)$ for any $\delta < 1/h$ and $\eps>0$. For that purpose, we need to show that for every $t\in \widetilde{E}_h \setminus F(\delta,\eps)$, the 2-microlocal frontier at $t$ satisfies $\sigma_{Y,t}(s') \leq (h+s')$. As $t\in E_h$, there exist two sequences $(v_n)_{n\in\N}$ and $(m_n)_{n\in\N}$ such that 
  \[
    \forall n\in\N;\quad v_n\in B\pthb{ t,2^{-m_n / (h+\eps)} }\quad\text{ and }\quad \normb{\Delta Y_{v_n}} \geq 2^{-m_n}.
  \]
  We may assume that $\eps$ is sufficiently small to satisfy $2^{-m_n / (h+\eps)} \leq 2^{-m\delta}$, i.e. $\delta < 1/(h+\eps)$. Furthermore, since $t\notin F(\delta,\eps)$, for every $m$ sufficiently large, there is no double-jump configuration in the neighbourhood of $t$ and $v_n$, meaning that $J\pthb{ B(v_n,2^{-m_n\delta}), D(2^{-m_n(1+\eps)},1) } = 1$: there does not exist other jump larger than $2^{-m_n(1+\eps)}$ in the neighbourhood of $v_n$. Therefore, we obtain the configuration presented in the proof of Lemma~\ref{lemma:2ml_levy2}, and as the latter remains valid, we have
  \[
    \forall s'\in\R;\quad \sigma_{Y,t}(s') \leq (h+s').
  \]
  This upper bound shows that $\widehat{E}_h \subset F(\delta,\eps)$, and considering the limits $\delta\rightarrow 1/h$ and $\eps\rightarrow 0$, it induces the inequality $\dimH\,\widehat{E}_h \leq 2\beta h - 1$. Furthermore, since $2\beta h - 1 < \beta h$ and $E_h=\widetilde{E}_h\cup\widehat{E}_h$, we have also proved that $\dimH\,\widehat{E}_h = \beta h$.\vsp

  To conclude this lemma, we obtain an upper bound of the 2-microlocal frontier in the case $t\in\widehat{E}_h$. Since the sketch of the proof is similar to Lemmas~\ref{lemma:2ml_levy1} and \ref{lemma:2ml_levy2}, we only present the main elements. Still using the previous two sequences $(v_n)_{n\in\N}$ and $(m_n)_{n\in\N}$, Lemma~\ref{lemma:tech_lemma3} induces that
  \[
    J\pthb{B(v_n, 2^{-m_n 2\beta(1+3\eps)}),D(2^{-m_n(1+\eps)},1)} = 1.
  \]
  Then, using the methodology presented in Lemma~\ref{lemma:2ml_levy2}, there exists $(s_n)_{n\in\N}$ such that for every $n\in\N$, $s_n\in B(t,2^{-m_n / (h+\eps)})$ and 
  \begin{align*}
    \norm{ Y(s_n,1) } 
    \geq c\, 2^{-m_n}\cdot 2^{-m_n 2\beta(1+3\eps)}
    \geq \abs{t-s_n}^{(h+\eps)(1+2\beta(1+3\eps))},
  \end{align*}
  Hence, $\alpha_{Y(\cdot,1),t} \leq (h+\eps)(1+2\beta(1+3\eps)) \rightarrow_{\eps\rightarrow 0} h(1+2\beta)$, and using the reasoning presented in Lemma~\ref{lemma:2ml_levy2}, we obtain $\sigma_{Y,t} \leq (h+s') / 2\beta h$ for all $s'\in\R$. 
\end{proof}

Before finally proving Theorem~\ref{th:2ml_levy} and its corollaries, we recall the following result on the increments of a Brownian motion. The proof can be found in \cite{Adler-1981} (inequality $(8.8.26)$).
\begin{lemma} \label{lemma:bm_incr}
  Let $B$ be a $d$-dimensional Brownian motion. There exists an event $\Omega_0$ of probability one such that for all $\omega\in\Omega_0$, $\eps>0$, there exists $h(\omega)>0$ such that for all $\rho\leq h(\omega)$ and $t\in\ivff{0,1}$, we have
  \[
    \sup_{u,v\in B(t,\rho)} \brcb{ \norm{B_u - B_v} } \geq \rho^{1/2 + \eps}.
  \]
\end{lemma}
\begin{proof}[Proof of Theorem \ref{th:2ml_levy}]
  We use the notations introduced at the beginning of the section. As previously said, the compound Poisson process $N$ can be ignored since it does not influence the final regularity. 
  Furthermore, if $Q = 0$, and therefore $B\equiv0$ and $\beta' = \beta$, Lemmas \ref{lemma:2ml_levy1}, \ref{lemma:2ml_levy2} and \ref{lemma:2ml_levy3} on the component $Y$ yields Theorem \ref{th:2ml_levy}. 
  
  Otherwise, the L\'evy process $X$ corresponds to the sum of the Brownian motion $B$ and the jump component $Y$. Still using Lemmas \ref{lemma:2ml_levy1}, \ref{lemma:2ml_levy2} and \ref{lemma:2ml_levy3}, it is sufficient to prove that with probability one and for all $t\in\ivff{0,1}$, $\sigma_{X,t} = \sigma_{B,t} \wedge \sigma_{Y,t}$. Owing to the definition of 2-microlocal frontier, we already know that $\sigma_{X,t} \geq \sigma_{B,t} \wedge \sigma_{Y,t}$. Furthermore, when $\sigma_{B,t}(s') \neq \sigma_{Y,t}(s')$, the upper bound is straightforward, and thus $\sigma_{X,t}(s') = \sigma_{B,t}(s') \wedge \sigma_{Y,t}(s')$. 

  Therefore, to obtain Theorem~\ref{th:2ml_levy}, we have to prove that with probability one, for all $t\in\ivff{0,1}$, $\sigma_{X,t} \leq \sigma_{B,t} = s'\mapsto\pthb{1/2 + s'}\wedge 1/2$. For that purpose, we distinguish two different cases.
  \begin{enumerate}[ \it 1.]
    \item If $\beta' = \beta = 2$, we only need to slightly modify the proof of Lemma~\ref{lemma:2ml_levy1}. More precisely, owing to Lévy's modulus of continuity, the increments of the Brownian motion satisfy:
    \[
      \forall u,v : \abs{u-v}\leq 2^{-m_n\gamma};\quad \norm{B_u-B_v} \leq c\,m_n 2^{-m_n\gamma/2} = c\,m_n 2^{-m_n(1+4\eps)},
    \]
    since $\gamma=\beta(1+4\eps)$. Therefore, the term due to the increments of the Brownian motion does not influence the rest of the estimates presented in the proof, ensuring that $\sigma_{X,t}(s') \leq \pth{ 1/2 + s' }\wedge 0$, for all $s'\in\R$.\vsp
    
    \item If $\beta < 2$, let $\delta = 2$ and $\eps > 0$. According to Lemma~\ref{lemma:tech_lemma3}, there exist $k\in\N$ and $M(\omega)\in\N$ such that for all $m\geq M$, there are at most $k$ jumps of size greater than $2^{-m(1+3\eps)}$ in any interval of size $2^{-m\delta}$. Hence, there always exists a sub-interval $I$ of size $c_0\,2^{-\delta m}$ with no jump greater than $2^{-m(1+3\eps)}$ inside.
    
    Still using Lemma~\ref{lemma:tech_lemma2}, we know that for all $m\geq M(\omega)$
    \[
      \forall u,v\in\ivff{0,1} : \abs{u-v}\leq 2^{-\delta m};\quad  \normb{ Y^{m(1+3\eps)}_u - Y^{m(1+3\eps)}_v } \leq c_1 m 2^{-m(1+3\eps)}.
    \]
    Let us set $t\in\ivff{0,1}$ and $I$ be one of the previous interval of size $c_0\,2^{-\delta m}$. According to Lemma~\ref{lemma:bm_incr}, there exist $u,v\in I$ such that $\norm{B_u-B_v}\geq c_0\,2^{-m(1+2\eps)}$. Then, 
    \begin{align*}
      \norm{ X_u - X_v } \geq \norm{B_u - B_v} - \normb{ Y^{m(1+3\eps)}_u - Y^{m(1+3\eps)}_v } - \abs{u-v}\cdot\normbb{ \int_{D(2^{-m(1+3\eps)},1)} x\, \pi(\dt x) },
    \end{align*}
    where $\abs{u-v}\cdot\normb{ \int_{D(2^{-m(1+3\eps)},1)} x\, \pi(\dt x) } \leq c \,2^{-m(1+3\eps)}$. Hence, we obtain a lower bound of the increments on the interval $I$, ensuring that the rest of the proof presented in Lemma~\ref{lemma:tech_lemma3} holds similarly.
  \end{enumerate}
\end{proof}
\begin{proof}[Proof of Corollary \ref{cor:2ml_levy_scalings}]
  Recall that $\beta^w_{X,t} = \lim_{s'\rightarrow-\infty} \sigma_{X,t}(s') - s'$. Hence, using the global upper bound on the 2-microlocal frontier proved in Theorem~\ref{th:2ml_levy}, we know that $\beta^w_{X,t}\leq 1 / \beta'$ with probability one. In addition, owing to the geometrical properties of the frontier, we observe that for every $h\in\ivff{0,1/\beta'}$
  \begin{equation*}
    \forall h\in\ivfo{0,1/\beta'};\quad \widetilde{E}_h  \subseteq E^w_h \subseteq \widetilde{E}_h \cup \bigcup_{h'<h} \widehat{E}_{h'}.
  \end{equation*}
  The first inclusion clearly shows that $\dimH(E^w_h\cap V)\geq \dimH(\widetilde{E}_h\cap V) = \beta h$. In addition, we also know that for every $h'<h$, $\dimH\,\widehat{E}_{h'} \leq 2\beta h'-1 < \beta h$, which proves the other side inequality.

  To obtain the upper bound on the oscillating exponent, we only need to note that according to its characterisation using the 2-microlocal frontier, 
  \begin{equation*}
    \forall h\in\ivfo{0,1/\beta'};\quad  \brcb{ t\in E_h : \beta^o_{X,t} > 0 } = \widehat{E}_{h}.
  \end{equation*}
  Finally, the chirp exponent is equal to one because of the upper bound $\sigma_{X,t}(s')\leq \pthb{1/\beta'+s'}$. 
\end{proof}
\begin{proof}[Proof of Corollary \ref{cor:2ml_levy}]
  Owing to upper bound on the 2-microlocal frontier obtained in Theorem~\ref{th:2ml_levy}, the case $\sigma = 0$ corresponds to the classic spectrum of singularity. Hence, let us set $\sigma<0$. We recall that $s$ denotes the parameter $\sigma-s'$. If $s\geq1/\beta'$ or $s<0$, the result is straight forward using Theorem \ref{th:2ml_levy} and properties of the 2-microlocal frontier.
  
  Therefore, we suppose that $s\in\ivfo{0,1/\beta'}$ and note that $E_{\sigma,s'} = \brc{t\in\R_+ : \sigma_{X,t}(s') = \sigma }$, since the negative component of the 2-microlocal frontier of $X$ can not be constant. Hence, similarly to the previous corollary, $E_{\sigma,s'}$ satisfies
  \begin{equation*}
    \forall s\in\ivfo{0,1/\beta'};\quad \widetilde{E}_s  \subseteq E_{\sigma,s'} \subseteq \widetilde{E}_s \cup \bigcup_{h<s} \widehat{E}_{h}.
  \end{equation*}
  These two inclusions lead to the same estimates, and therefore the expected equality on the Hausdorff dimension.
\end{proof}


\subsection{Oscillating singularities of some classes of Lévy processes}  \label{ssec:2ml_alpha}

In this section, we aim to understand more precisely the oscillating singularities of Lévy processes captured by the collection of sets $(\widehat{E}_h)_{h\in\R_+}$. Note that to simply our presentation, we assume that $d=1$.

Let us begin with the proof of Proposition~\ref{prop:2ml_levy_oneside} where we present a class Lévy processes with no chirp oscillations. Recall that in this case, we consider Lévy measures such that $\pi(\R_\pm) = 0$.

\begin{proof}[Proof of Proposition~\ref{prop:2ml_levy_oneside}]
  In order to prove that $\widehat{E}_h = \emptyset$ for all $h\in\R_+$, we extend Lemma~\ref{lemma:2ml_levy2} to any $h\in\ivfo{0,1/\beta}$. We may assume without any loss of generality that $\pi(\R_-) = 0$. We still consider the two sequences $(v_n)_{n\in\N}$ and $(m_n)_{n\in\N}$ such that 
  \[
    \forall n\in\N;\quad v_n\in B\pthb{ t,2^{-m_n / (h+\eps)} }\quad\text{ and }\quad \absb{\Delta Y_{v_n}} \geq 2^{-m_n}.
  \]
  where we suppose that $v_n\geq t$ and $Y$ designates the jump component. In addition, we first assume that $\beta\geq 1$ and $Y(v_n,0)\geq 2^{-m_n-1}$, and we set $t_n = v_n$ and $\rho_n = 2^{-m_n\delta}$. Then, since the Lévy process only has positive jumps, for all $u\in\ivfo{t_{n},t_{n}+\rho_{n}}$,
  \begin{align*}
    \abs{Y(u,0)} \geq \abs{Y(v_n,0)} - \absb{ Y_u^{m_n(1+\eps)}-Y_{v_n}^{m_n(1+\eps)} } - \absbb{(u-v_n) \int_{D(2^{-m_n(1+\eps)},1)} x \,\pi(\dt x)},
  \end{align*}
  According to the proof presented in Lemma~\ref{lemma:2ml_levy2}, this inequality is sufficient to show that $\sigma_{Y,t}(s')\leq\pth{\alpha_{X,t}+s'}\wedge 0$. The cases $Y(v_n,0)\leq 2^{-m_n-1}$ and $\beta \leq 1$ are then treated similarly, proving that the 2-microlocal frontier of the process $X$ is equal to $\pth{\alpha_{X,t}+s'}\wedge 0$.
\end{proof}
Proposition~\ref{prop:2ml_levy_oneside} proves in particular that Lévy subordinators, in which case $\beta \leq 1$, only have cusp singularities.\vsp

The second important class of Lévy processes we consider are characterised by the following Lévy measure
\begin{align*}
  \pi(\dt x) = a_1 \,\abs{x}^{-1-\alpha_1} \,\indi_{\R_+}\dt x + a_2 \,\abs{x}^{-1-\alpha_2} \,\indi_{\R_-}\dt x,
\end{align*}
where $a_1,a_2 > 0$ and $\alpha_1,\alpha_2\in\ivoo{0,2}$. 

The proof of Theorem~\ref{th:2ml_alpha} is rather technical and will be divided in several parts for the sake of readability. To begin with, we present two simple technical lemmas related to the Binomial distribution. Recall that Chernoff's inequality states that for any $\eps\in\ivoo{0,1}$,
\begin{equation}  \label{eq:ineq_chernoff1}
  \prb{ N \leq np(1-\eps) } \leq \exp\pthb{-np \,\eps^2 / 2}.
\end{equation}
and
\begin{equation}  \label{eq:ineq_chernoff2}
  \prb{ N \geq np(1+\eps) } \leq \exp\pthb{-np \,\eps^2 / 2}.
\end{equation}
where $N$ follows a Binomial distribution of parameters $n$ and $p$.
\begin{lemma}  \label{lemma:binomial1}
  Suppose $N$ follows a Binomial distribution with parameters $n$ and $p$. Then, there exists $c>0$ such that when $n>c$ and $p<1/c$,
  \begin{align*}
    &\prB{ \#\brcb{\text{empty intervals of size} \geq 1/p \LL(1/p) } \geq np \LL(1/p) h(p)^4 } \\
    &\geq 1 - \exp\pthb{-np / 8 \LL(1/p)^2},
  \end{align*}
  using the notations: $\LL(1/p)\eqdef\log(\log(1/p))$ and $h(p) \eqdef\exp(-1/\LL(p))$.
\end{lemma}
\begin{proof}
  To obtain this lower bound, we first estimate the probability of obtaining an empty interval of size of $l_0 = \floor{1/p\LL(1/p)}$. Setting $p_0\eqdef (1-p)^{l_0}$, we note that
  \[
    \log(p_0) \geq \frac{1}{p\LL(1/p)} \log(1-p) \geq \frac{-2}{\LL(1/p)}.
  \]
 Therefore, $p_0 \geq h(p)^2$  when $c$ is sufficiently large. 

  Let $n_0$ denote the number of disjoint sub-intervals of size $l_0$ and $N_0$ be a r.v. following a Binomial distribution of parameters $n_0$ and $p_0$. Owing to Chernoff's inequality,
  \[
    \prb{ N_0 \geq n_0p_0 \,h(p) } \geq 1 - \exp\pthb{ -n_0p_0 \,g(p)^2 /2 },
  \]
  where $g(p)\eqdef 1-h(p)$. Note that $n_0 \geq np \LL(1/p) h(p)$ and $g(p)\asymp 1/\LL(1/p)$ when $c$ is large enough. Therefore, using the previous estimates,
  \[
    \prb{ N_0 \geq np \LL(1/p) h(p)^4} \geq 1 - \exp\pthb{ -np / 8 \LL(1/p)^2 },
  \]
  proving the lemma.
\end{proof}

\begin{lemma}  \label{lemma:binomial2}
  Suppose $N$ follows a Binomial distribution with parameters $n$ and $p$. Then, there exists $c>0$ such that when $n>c$ and $p<1/c$
  \[
    \prB{ \#\brcb{\text{successes spaced by} \geq 1/p } \geq np / 6 } \geq 1 - \exp\pthb{-np/16}.
  \]
\end{lemma}
\begin{proof}
  The sketch of the proof is similar to Lemma~\ref{lemma:binomial1}. The set $\brc{1,\ldots,n}$ can be divided in $n_0$ intervals of size $l_0 = \ceil{1/p}$. The probability $p_0$ of obtaining at least a success in one of these intervals is equal to:
  \[
    p_0 = 1 - (1-p)^{1/p} \longrightarrow_{ p\rightarrow 0} 1 - \e^{-1} \geq 1/2.
  \]
  Furthermore, only considering one third of the previous intervals, i.e. $n_0/3$, we consider the Binomial distribution $B(n_0/3, p_0)$. Still using Chernoff's inequality, we obtain
  \[
    \prb{ N_0 \geq np / 6 } \geq 1 - \exp\pthb{ -np/16 },
  \]
  which concludes the proof.
\end{proof}

\begin{proof}[Proof of Theorem \ref{th:2ml_alpha}]
  As observed in the proofs of Theorem~\ref{th:2ml_levy} and Proposition~\ref{prop:2ml_levy_oneside}, chirp singularities appear when a compensation phenomena between jumps exists. Hence, the main goal of the proof is to characterise in more details this particular behaviour in the case of the Lévy measure considered. 
  Firstly, we clearly note the Blumenthal--Getoor exponent $\beta$ of $\pi$ is equal to $\max\pth{\alpha_1,\alpha_2}$. \vsp
 
  \noindent\textbf{Hausdorff dimension (upper-bound).} To obtain a tighter upper bound of the Hausdorff dimension, we need to enhance the estimates presented in Lemma~\ref{lemma:2ml_levy3}. We have observed in the proof of Proposition~\ref{prop:2ml_levy_oneside} that oscillating singularities do not appear when there are jumps of the same sign. Hence, we are interested in the \emph{double-jump} configurations with jumps of opposite signs.

  Suppose $\delta\in\ivoo{\beta,\alpha_1+\alpha_2}$, $\eps>0$ and $I$ is an interval of size $2^{-j\delta+1}$. We are interested in the following type of configurations: $J\pthb{I, \ivof{2^{-j(1+\eps)}, 1}} \geq 1$ and $J\pthb{I, \ivfo{-1,-2^{-j(1+\eps)}}} \geq 1$. The probability $p_j$ of such an event satisfies:
  \[
    p_j \asymp 2^{-j\delta}2^{j\alpha_1(1+\eps)}\cdot 2^{-j\delta}2^{j\alpha_2(1+\eps)} = 2^{-j2\delta+j(\alpha_1+\alpha_2)(1+\eps)}.
  \]
  Using this probability, the rest of the proof is rather similar to Lemma~\ref{lemma:2ml_levy3}. We consider $\ceil{2^{j\delta}}$ consecutive, but disjoint, intervals of size $2^{-j\delta+1}$ sufficient to cover $\ivff{0,1}$ and we denote by $N^1_j$ the number of intervals with the previous configuration. Owing to Chernoff's inequality,
  \[
    \prb{ N^1_j \geq c_0\, 2^{-j\delta+j(\alpha_1+\alpha_2)(1+\eps)} } \leq \exp\pthb{-c \,2^{-j\delta+j(\alpha_1+\alpha_2)} },
  \]
  where $(\alpha_1+\alpha_2)>\delta$. Consider now the same configurations of intervals translated by $2^{-j\delta}$ and denote by $N^2_j$ the corresponding Binomial random variable. Using Borel--Cantelli lemma, with probability one, there exists $M(\omega)$ such that for every $j\geq M(\omega)$, $N^1_j \leq c_0\, 2^{-j\delta+j(\alpha_1+\alpha_2)(1+\eps)}$ and $N^2_j \leq c_0\, 2^{-j\delta+j(\alpha_1+\alpha_2)(1+\eps)}$. Using the same notation $T_j$, we define
  \[
    F(\delta,\eps) = \limsup_{j\rightarrow\infty} \bigcup_{I\in T_j} \ivffb{c(I)-2^{-j\delta+2},c(I)+2^{-j\delta+2}},
  \]
  where $c(I)$ denotes the center of any interval $I\in T_j$. Then, for any $j_0\in\N$
  \[
    \sum_{j=j_0}^{+\infty} c\,\abs{T_j} \cdot \pthb{2^{-j\delta}}^\gamma \leq c \sum_{j=j_0}^{+\infty} 2^{-j(\delta(1+\gamma)-(\alpha_1+\alpha_2)(1+\eps))},
  \]
  is finite when $\delta(1+\gamma)>(\alpha_1+\alpha_2)(1+\eps)$, i.e. $\gamma > (\alpha_1+\alpha_2)(1+\eps)/\delta - 1$. Therefore, $\dimH\, F(\delta,\eps) \leq (\alpha_1+\alpha_2)(1+\eps)/\delta - 1$.

  The rest of the proof of Lemma~\ref{lemma:2ml_levy3} does not change, proving that for any $\delta < 1 / h$ and $\eps>0$, $\widehat{E}_h\subset F(\delta,\eps)$. Therefore, with probability one,
  \[
    \forall h\in\ivoob{1/(\alpha_1 + \alpha_2), 1/\beta};\quad \dimH \,\widehat{E}_h \leq (\alpha_1+\alpha_2)h - 1.
  \]
  Finally, when $h\notin\ivoob{1/(\alpha_1 + \alpha_2), 1/\beta}$, the proof of Lemma~\ref{lemma:2ml_levy2} can also be similarly adapted to prove that $\widehat{E}_h=\vset$ with probability one.
  \vsp

  \noindent\textbf{Construction (lower-bound).} In order to prove the lower bound of the Hausdorff dimension, we need to construct a proper set of times with singularities.
  
  For our construction procedure, we will need a set of parameters $\bm{p} = (\delta,\delta',\delta'',\gamma,\rho)$ such that $\delta'<\delta<\delta''\in\ivoo{\beta,\alpha_1+\alpha_2}$, $\delta<\gamma\in\ivoo{\beta,\alpha_1+\alpha_2}$, $\delta' < \sqrt{\beta\delta}$ and $\delta > \sqrt{\delta'\delta''}$. In addition, we also define the sequence $j_n = (\delta/\delta')^n\rightarrow\infty$.\vsp

  The first step consists in constructing collections of intervals such that for every $t$ inside, there is no jump of size $2^{-j}$ or greater closer than $2^{-j\delta}$ for all $j\leq j_0$, where $j_0$ is a given index. More precisely, we define by induction a collection, indexed by the random variables $\Si_n$, of disjoint intervals of size $2^{-j_n\delta}$ in the following way. Suppose $\Si_n$ is defined such that for every $t$ inside an interval, there is no jump greater than $2^{-j_{n+1}}$ closer than $2^{-j_{n}\delta}$ of $t$. In every interval of size $2^{-j_n\delta}$, we consider consecutive sub-intervals of size $2^{-j_{n+1}\delta}$ with no jumps greater than $2^{-j_{n+2}}$ inside. Removing the left and right elements of these collections, we obtain the family $\Si_{n+1}$, which corresponds to the offspring of $\Si_{n}$. Owing to this construction procedure, we know that the remaining intervals satisfy the expected property, i.e. for any $t$ inside, there is no jump greater than $2^{-j_{n+2}}$ closer than $2^{-j_{n+1}\delta}$. 

  In order to determine the number of this type of intervals, we need to estimate the law of $\abs{\Si_{n+1}}$ conditionally to $\abs{\Si_{n}}$. For any $n\in\N$, let us denote by $p_{n+1}$ the probability of obtaining at least one jump greater $2^{-j_{n+2}}$ inside an interval of size $2^{-j_{n+1}\delta}$. Note that for every $n$ sufficiently large, an interval of size $2^{-j_n\delta}$ can be divided in at least $2^{-j_n \delta} / 2^{-j_{n+1}\delta} h( p_{n+1} )= 2^{j_{n+1}(\delta-\delta') } h( p_{n+1} )$ sub-intervals. Furthermore, let $M_{n+1}$ designate the following random variable:
  \[
    M_{n+1} = \#\brcb{\text{family $\geq 1/p_{n+1} \LL(1/p_{n+1})$ of consecutive empty intervals of size } 2^{-j_{n+1}} }.
  \]
  According to Lemma~\ref{lemma:binomial1},
  \begin{align*}
    &\prcb{ M_{n+1} \geq s_0 \,2^{j_{n+1}\delta } p_{n+1} \LL(1/p_{n+1}) h(p_{n+1})^5 }{ \abs{\Si_{n}} \geq s_0 \,2^{j_{n}\delta} } \\
    &\geq 1 - \exp\pthb{-s_0 \,2^{j_{n+1}\delta } p_{n+1} / 8 \LL(1/p_{n+1})^2},
  \end{align*}
  for any $s_0\in\R_+$ such that $s_0 \,2^{j_{n}\delta} \geq 1$.
  As previously outlined, for every collection of consecutive empty intervals, we remove the extremal elements to constitute the family $\Si_{n+1}$. Noting that $1 / p_{n+1} \LL(1/p_{n+1}) - 2 \geq h(p_{n+1}) / p_{n+1} \LL(1/p_{n+1})$ for any $n$ sufficiently large, we therefore obtain
  \begin{align*}
    &\prcb{ \abs{\Si_{n+1}} \geq s_0 \,2^{j_{n+1}\delta } h(p_{n+1})^6 }{ \abs{\Si_{n}} \geq s_0 \,2^{j_{n}\delta} } \\
    &\geq 1 - \exp\pthb{-s_0 \,2^{j_{n+1}\delta } p_{n+1} / 8 \LL(1/p_{n+1})^2}.
  \end{align*}
  Furthermore, the probability of obtaining an empty interval of size $2^{-j_{n+1}\delta}$ is equal to:
  \begin{align*}
    q_{n+1} 
    &\eqdef \prb{ J\pthb{ \ivff{0,2^{-j_{n+1}\delta}}\times D\pthb{2^{-j_{n+2}},2^{-j_{n+1}}} } = 0 } \\
    &\asymp \exp\pthb{ -c\, 2^{j_{n+2}\beta -j_{n+1}\delta} } 
    = \exp\pthb{ -c \,2^{-j_{n+1}\delta(1 - \beta/\delta') } } \longrightarrow 1.
  \end{align*}
  Hence, $p_{n+1}=1-q_{n+1} \asymp 2^{-j_{n+1}\delta(1 - \beta/\delta') }$ for any $n$ sufficiently large, and there exists $c_1>0$ such that 
  \begin{align*}
    \prcb{ \abs{\Si_{n+1}} \geq s_0 \,2^{j_{n+1}\delta } h(p_{n+1})^6 }{ \abs{\Si_{n}} \geq s_0 \,2^{j_{n}\delta} } 
    \geq 1 - \exp\pthb{-c_1 \,s_0 \,2^{j_{n+1}\beta\delta/\delta' } / (n+1)^2 }.
  \end{align*}
  Furthermore, note $2^{j_{n+1}\beta\delta/\delta'-j_n\delta} = 2^{j_{n+1}(\beta\delta/\delta'-\delta')}$. Since we have assumed that $\delta'<\sqrt{\beta\delta}$ and $s_0 \,2^{j_{n}\delta}\geq 1$, there exists $r>0$ such that 
  \begin{align*}
    \prcb{ \abs{\Si_{n+1}} \geq s_0 \,2^{j_{n+1}\delta } h(p_{n+1})^6 }{ \abs{\Si_{n}} \geq s_0 \,2^{j_{n}\delta} } 
    \geq 1 - \exp\pthb{-2^{j_{n+1}r } }.
  \end{align*}
  Therefore, by induction, the law of $\abs{\Si_{n+m}}$ satisfies
  \begin{align*}
    \prcbb{ \abs{\Si_{n+m}} \geq s_0 \, 2^{j_{n+m}\delta} \prod_{k=1}^m h(p_{n+k})^6 }{ \abs{\Si_n} \geq s_0\,2^{j_{n}\delta} } 
    \geq \prod_{k=1}^m \pthB{ 1 - \exp\pthb{-2^{j_{n+k}r } } }.
  \end{align*}
  Finally, we note that $\prod_{k=1}^m h(p_{n+k})^6 \geq \exp\pthb{ -c \sum_{k=1}^m 1/(n+k) } \geq \exp\pthb{-c \log(n+m)}$, implying that
  \begin{align}
    \prcB{ \abs{\Si_{n+m}} \geq s_0 \, 2^{j_{n+m}\delta -c_2 \log(n+m) } }{ \abs{\Si_n} \geq s_0\,2^{j_{n}\delta} } 
    \geq \prod_{k=1}^m \pthB{ 1 - \exp\pthb{-2^{j_{n+k}r } } },
  \end{align}
  where $c_2$ is a constant independent of $n$ and $m$.\vsp

  The previous bound gives us an estimate of the probability of obtaining intervals without any jump in a given neighbourhood. Using this estimate, we will be able to construct our main collection of nested intervals indexed by $(T_\ell)_{\ell\in\N}$ such that a most scales $2^{-j_n}$, there is no jump in the neighbourhood, and at specific ones $2^{-j_{n(\ell)}}$, a particular double-jump configuration appears. To construct this collection, let us first define this sequence $\pthb{n(\ell)}_{\ell\in\N}$: 
  \[
    n(0) = 1 \quad\text{ and }\quad n(\ell+1) = 2^{n(\ell)}\quad \forall \ell\in\N.
  \]
  For every $\ell\in\N$, we are interested in the following type of configuration: in an interval of size $2^{-j_{n(\ell)}\delta} / 3$, there exist two jumps $\Delta X_{u}$ and $\Delta X_{v}$ of opposite sign inside the middle third and such that $\abs{u-v} \leq 2^{-j_{n(\ell)}\gamma}$, $\abs{\Delta X_{u}},\abs{\Delta X_{v}}\in\ivff{2^{-j_{n(\ell)}-1},2^{-j_{n(\ell)}}}$ and $\abs{\Delta X_{v}-\Delta X_{u}} \leq 2^{-j_{n(\ell)}\rho}$. Using the independence property on the Poisson measure $J$, the probability $r_{\ell}$ of the previous event can be lower bounded by
  \begin{align*}
    r_{\ell} 
    &\asymp \exp\pthb{ -c\, 2^{-j_{n(\ell)}\delta+j_{n(\ell)+1}\beta } } \cdot 2^{-j_{n(\ell)}(\delta-\alpha_1)} \cdot 2^{-j_{n(\ell)}(\gamma-\alpha_2-1+\rho)} \\
    &\asymp \, 2^{-j_{n(\ell)}(\delta+\gamma-\alpha_1-\alpha_2-1+\rho)}.
  \end{align*}
  
  The collection of intervals $T_\ell$ is constructed by induction. $T_0$ is initialised with the singleton corresponding to the interval $\ivff{0,1}$. Then, assuming $T_\ell$ is defined, for any $I\in T_{\ell}$, we consider the sub-intervals of size $2^{-j_{n(\ell+1)}\delta}$ with a double-jump configuration and with no jump in the neighbourhood at all intermediate scales $2^{-j_{n}}$, $n(\ell) < n < n(\ell+1)$. Note that if none satisfy the previous conditions, we avoid the extinction of the tree by selecting a single sub-interval of size $2^{-j_{n(\ell+1)}\delta}$.

  We aim to estimate the size of $T_{\ell+1}$ conditionally to $T_{\ell}$. For any $I\in T_{\ell}$, we denote by $c(I)$ the middle point between the two jump times inside $I$. Then, for every integer $k\in\ivff{j_{n(\ell)}\delta,j_{n(\ell)}\delta''}$, we want to estimate the number double-jumps configurations inside the interval of size $2^{-k}$:
  \begin{equation}  \label{eq:interval_Ik}
    I_k \eqdef \ivfob{c(I)-2^{-k}, c(I)-2^{-k-1}}\cup\ivofb{c(I)+2^{-k-1}, c(I)+2^{-k}}.
  \end{equation}
  We designate by $\Si_{n(\ell)+1,k}$ the number of sub-intervals of size $2^{-j_{n(\ell)+1}\delta}$ inside $I_k$ which are empty. Using Chernoff's inequality, the latter satisfies
  \begin{align*}
    \prb{ \abs{\Si_{n(\ell)+1,k}} \geq  2^{j_{n(\ell)+1}\delta-k-2 } }
    \geq 1 - \exp\pthb{ -2^{j_{n(\ell)+1}\delta-k-4} } \geq 1-\exp\pthb{ -2^{j_{n(\ell)+1}r} }.
  \end{align*}
  The last inequality is due to $j_{n(\ell)+1}\delta-k\geq j_{n(\ell)+1}\delta-j_{n(\ell)}\delta''>j_{n(\ell)+1}r$, if $r$ is sufficiently small. Therefore, using the estimates obtained previously,
  \begin{align*}
    \prB{ \abs{\Si_{n(\ell+1)-1,k}} \geq 2^{j_{n(\ell+1)-1}\delta - c_3 n(\ell)-k } } 
    &\geq \prod_{k=1}^{n(\ell+1)-1} \pthB{ 1 - \exp\pthb{-2^{j_{n(\ell)+k}r } } } \\
    &\geq \pthb{ 1 - \exp\pthb{-2^{j_{n(\ell)}r } } }^{n(\ell+1)},
  \end{align*}
  where $c_3$ is a positive constant and we recall that $n(\ell+1) = 2^{n(\ell)}$. An interval of size $2^{-j_{n(\ell+1)-1}\delta}$ can be divided in at least $2^{j_{n(\ell+1)}(\delta-\delta')-1}$ sub-intervals. Hence, if $M_{\ell,k}$ denotes the number of double-jump configurations existing among the sub-intervals of $\Si_{n(\ell+1)-1,k}$, Lemma~\ref{lemma:binomial2} and the estimate of $r_\ell$ induce that 
  \begin{align*}
    &\prB{ \abs{M_{\ell,k}} \geq 2^{-j_{n(\ell+1)}(\delta+\gamma-\alpha_1-\alpha_2-1+\rho)} \cdot 2^{j_{n(\ell+1)}\delta - c_3 n(\ell)-k - 4} } 
    \geq \pthb{ 1 - \exp\pthb{-2^{j_{n(\ell)}r } } }^{n(\ell+1)+1}.
  \end{align*}
  We observe that the intervals $I_k$ are disjoints for different integers $k$. Hence, the probability of the intersection of the previous event for every $k\in\ivff{j_{n(\ell)}\delta,j_{n(\ell)}\delta''}$ satisfies
  \begin{align*}
    &\prbb{ \bigcap_{k\geq j_{n(\ell)}\delta}^{j_{n(\ell)}\delta''} \abs{M_{\ell,k}} \geq 2^{j_{n(\ell+1)}(\alpha_1+\alpha_2-\gamma+1-\rho) - c_4 n(\ell)-k} } \\
    &\geq \pthb{ 1 - \exp\pthb{-2^{j_{n(\ell)}r } } }^{(n(\ell+1)+1) 2\beta j_{n(\ell)} },
  \end{align*}
  since we assume that $\delta'' \leq \alpha_1+\alpha_2\leq 2\beta$.
  The previous construction procedure leads to estimate of size of $T_{\ell+1}$. Therefore, conditionally to the event $\brcb{ \abs{T_\ell} \geq k_0 2^{j_{n(\ell)}(\alpha_1+\alpha_2-\gamma+1-\rho) - 2j_{n(\ell-1)}\delta} }$, we obtain
  \begin{align*}
    &\prcB{ \abs{T_{\ell+1}} \geq k_0 2^{j_{n(\ell+1)}(\alpha_1+\alpha_2-\gamma+1-\rho) - 2j_{n(\ell)}\delta} }{ \abs{T_\ell} \geq k_0 2^{j_{n(\ell)}(\alpha_1+\alpha_2-\gamma+1-\rho) - 2j_{n(\ell-1)}\delta} } \\
    &\geq \pthb{ 1 - \exp\pthb{-2^{j_{n(\ell)}r } } }^{(n(\ell+1)+1)\,2\beta j_{n(\ell)} \,2^{j_{n(\ell) c}} }
    \geq \pthb{ 1 - \exp\pthb{-2^{j_{n(\ell)}r } } }^{2^{j_{n(\ell)}c_5} }.
  \end{align*}
  For any $\ell_0\in\N$, we know that the construction ensures that $\abs{T_{\ell_0}}\geq 1$ almost surely. Hence, choosing $k_0 = 2^{-c_6 j_{n(\ell_0)}}$, with the proper constant $c_6$, we obtain that the following lower bound
  \begin{align*}
    &\prbb{ \bigcap_{\ell>\ell_0} \abs{T_{\ell}} \geq 2^{j_{n(\ell)}(\alpha_1+\alpha_2-\gamma+1-\rho) - 2j_{n(\ell-1)}\delta - c_6 j_{n(\ell_0)}  } } \geq \prod_{\ell>\ell_0} \pthb{ 1 - \exp\pthb{-2^{j_{n(\ell-1)}r } } }^{2^{j_{n(\ell-1)}c_5} }.
  \end{align*}
  Considering the logarithm of the right term, we observe 
  \begin{align*}
    \sum_{\ell>\ell_0} \log\pthb{ 1 - \exp\pthb{-2^{j_{n(\ell-1)}r } } }^{2^{j_{n(\ell-1)}c_5} } \geq -\sum_{\ell>\ell_0} 2^{j_{n(\ell-1)}c_5}\cdot\exp\pthb{-2^{j_{n(\ell-1)}r } } \longrightarrow_{\ell_0\rightarrow\infty} 0.
  \end{align*}
  Hence, the previous probability converges to $1$ for any set of parameters $\bm{p}=(\delta,\delta',\delta'',\rho)$. Since the family of events considered is increasing with $\ell_0$, it implies that almost surely there exists $\ell_0(\omega)$ such that $\abs{T_{\ell}} \geq 2^{j_{n(\ell-1)}(\alpha_1+\alpha_2-\gamma+1-\rho) - 2j_{n(\ell)}\delta - c_6j_{n(\ell_0)} }$ for all $\ell>\ell_0(\omega)$. Furthermore, owing to the construction procedure described previously, we also know that for any $I\in T_\ell$, every interval $I_{k}$ defined in Equation~\eqref{eq:interval_Ik} contains at least $2^{j_{n(\ell+1)}(\alpha_1+\alpha_2-\gamma+1-\rho) - c_4 n(\ell)-k}$ proper double-jump configurations.\vsp

  \noindent\textbf{Hausdorff dimension (lower-bound).} The previous estimates now allow us to study more precisely the oscillating behaviour of the Lévy process. Suppose $h\in\ivoo{1/(\alpha_1+\alpha_2),1/\beta}$ and $\bm{p}$ is a set of parameters such that $\delta < 1/h < \delta''$ and $1/h < \gamma$. Then, let us define the set of interest $G(h,\bm{p})$ as following:
  \begin{align}  \label{eq:set_Ghp}
    G(h,\bm{p}) 
    = \bigcap_{\ell>\ell_0} \bigcup_{I\in T_{\ell}} \brcB{ &\ivffb{ c(I)-2^{-j_{n(\ell)}/h+1},c(I)-2^{-j_{n(\ell)}/h-1} } \nonumber \\
    \cup &\ivffb{ c(I)+2^{-j_{n(\ell)}/h-1},c(I)+2^{-j_{n(\ell)}/h+1} } },
  \end{align}
  where $c(I)$ still denotes the middle point of any double-jump interval $I\in T_{\ell}$ and $\ell_0(\omega)$ corresponds to the random index previously defined. Owing to the construction of the tree $T$, we note that $G(h,\bm{p})$ corresponds to to the intersection of collections $T_{\ell,h}$ of nested intervals of size $3\cdot2^{-j_{n(\ell)}/h-1}$. Furthermore, according to the estimates obtained in the previous paragraph, we know that every $I\in T_{\ell,h}$ contains at least $2^{j_{n(\ell+1)}(\alpha_1+\alpha_2-\gamma+1-\rho) - c_4 n(\ell)-j_{n(\ell)}/h-1}$ sub-elements separated by $2^{j_{n(\ell+1)}(\gamma-\alpha_1-\alpha_2-1+\rho)}$.

  In order the estimate the Hausdorff dimension of the set $G(h,\bm{p})$, we construct by induction a mass measure $\mu$ on it. To begin with, $\mu_{\ell_0}$ attributes an equivalent weight on every interval $I\in T_{\ell_0,h}$. Then, similarly to the procedure on Cantor's set, $\mu_{\ell+1}$ is defined on the intervals $I\in T_{\ell+1,h}$ such that the weight $\mu_{\ell}(\Ii)$, $\Ii\in T_{\ell,h}$, is equally distributed on its offspring. The measure $\mu$ is then defined as  the limit of the sequence $(\mu_{\ell})_{\ell\geq\ell_0}$, which clearly exists since the cumulative distribution functions uniformly converge on $\ivff{0,1}$. 

  Since every $I\in T_{\ell,h}$ contains at least $2^{j_{n(\ell+1)}(\alpha_1+\alpha_2-\gamma+1-\rho) - c_4 n(\ell)-j_{n(\ell)}/h}$ elements,
  \begin{align*}
    \forall I\in T_{\ell,h};\quad\mu(I) 
    &\leq \prod_{k=\ell_0+1}^\ell 2^{-j_{n(k)}(\alpha_1+\alpha_2-\gamma+1-\rho) + c_4 n(k-1)+j_{n(k-1)}/h } \\
    &\leq 2^{-j_{n(\ell)}(\alpha_1+\alpha_2-\gamma+1-\rho) + c_7 j_{n(\ell-1)}},
  \end{align*}
  since we note that $n(\ell) \leq j_{n(\ell)}$ and $\sum_{k=1}^\ell j_{n(k)} \leq c\, j_{n(\ell)}$ for any $\ell\in\N$.

  As we aim to use the usual \emph{mass distribution principle} to determine a gauge function $g$ such that the Hausdorff measure $\Hi_g$ of $G(h,\bm{p})$ is positive, we need to obtain an upper bound of $\mu\pthb{ B(t,r) }$ for any $t\in\ivff{0,1}$ and $r>0$ sufficiently small. There exists $\ell\in\N$ such that $2^{-j_{n(\ell)}/h} \leq r < 2^{-j_{n(\ell-1)}/h}$ and without any loss of generality, we may assume that $\ell > \ell_0$. Furthermore, as $r < 2^{-j_{n(\ell-1)}/h}$, we may also suppose that $B(t,r)\subset \Ii$ where $\Ii\in T_{\ell-1,h}$ (otherwise, consider the intersection $B(t',r')$ between $B(t,r)$ and the closest element $\Ii$). 
  Since the sub-intervals $I\in T_{\ell,h}$, with $I\subset\Ii$ are separated by at least $2^{j_{n(\ell+1)}(\gamma-\alpha_1-\alpha_2-1+\rho)}$, we know that the ball $B(t,r)$ intersects with at most $r \, 2^{j_{n(\ell)}(\gamma-\alpha_1-\alpha_2-1+\rho)+1}$ of them. Hence, since $\mu(I)$ has the same value for every $I\in T_{\ell,h}$ with $I\subset \Ii$, we obtain
  \begin{align*}
    \mu\pthb{ B(t,r) } 
    &\leq r\, 2^{j_{n(\ell)}(\gamma-\alpha_1-\alpha_2-1+\rho)+1} \mu\pth{I} \quad \\
    &\leq r\,  2^{j_{n(\ell-1)}/h + c_4 n(\ell-1) + 1} \mu\pth{\Ii},
  \end{align*}
  as $\mu\pth{I}\leq \mu\pth{\Ii} / 2^{j_{n(\ell)}(\alpha_1+\alpha_2-\gamma+1-\rho) - c_4 n(\ell-1)-j_{n(\ell-1)}/h}$. In addition, we know that $\mu\pth{\Ii}\leq 2^{-j_{n(\ell-1)}(\alpha_1+\alpha_2-\gamma+1-\rho) + c_7 j_{n(\ell-2)}}$ and $j_{n(\ell-2)} \leq n(\ell-1)$, inducing 
  \[
    \mu\pthb{ B(t,r) } \leq r\, 2^{-j_{n(\ell-1)}(\alpha_1+\alpha_2-\gamma+1-\rho) + j_{n(\ell-1)}/h + c_8 n(\ell-1) },
  \]
  Furthermore, as $\gamma>\beta$, $1/h>\beta$ and $\rho>1$, $\alpha_1+\alpha_2-\gamma-1/h+1-\rho < 0$, $1 - (\alpha_1+\alpha_2-\gamma+1-\rho)h > 0$ and $r^{1 - (\alpha_1+\alpha_2-\gamma+1-\rho)h} \leq 2^{-j_{n(\ell-1)}/h + j_{n(\ell-1)}(\alpha_1+\alpha_2-\gamma+1-\rho))}$. Finally, since $j_{n(\ell-1)} = \pthb{\delta/\delta'}^{n(\ell-1)} \leq c \log\pthb{1/r}$, there exist $c_{9},c_{10} > 0$ such that for all $t\in\ivff{0,1}$ and $r>0$
  \begin{align*}
    \mu\pthb{ B(t,r) } \leq c_{9} \log\pth{1/r}^{c_{10}} \, r^{(\alpha_1+\alpha_2-\gamma+1-\rho)h}.
  \end{align*}
  Using the mass distribution principle (see \cite{Falconer-1986} for instance), this inequality proves that $G(h,\bm{p})$ has a positive $g$-Hausdorff measure, where the gauge function $g$ is defined by $g(r) = \log\pth{1/r}^{c_{10}} \,r^{(\alpha_1+\alpha_2-\gamma+1-\rho)h}$. 

  Therefore, if we restrict ourselves to rational parameters $\bm{p}$, we have proved that with probability one, for all $h\in\ivoo{1/(\alpha_1+\alpha_2),1/\beta}$, $\dimH\, G(h,\bm{p}) \geq (\alpha_1+\alpha_2-\gamma+1-\rho)h$. \vsp
  
  \noindent\textbf{2-microlocal frontier (lower-bound).} In this last step of the proof, we aim to show that the 2-microlocal frontier of every $t\in G(h,\bm{p})$ has a chirp oscillation shape. 

  Let us set $\omega\in\Omega$ and $t\in G(h,\bm{p})$. As previously outlined in this work, we know that we may ignore the component of Lévy process which corresponds to the jumps of size greater than $2^{-j_{n(\ell_0)}}$. Furthermore, owing to the construction of the set $G(h,\bm{p})$, we know that for every $\ell\in\N$, the distance between $t$ and the closest jump time $s$ such that $\abs{\Delta X_s}\in\ivff{2^{-j_{n(\ell)}-1},2^{-j_{n(\ell)}}}$, satisfies
  \[
     2^{-j_{n(\ell)}/h-2} \leq \abs{s-t} \leq 2^{-j_{n(\ell)}/h+1}. 
  \]
  Therefore, owing to the characterisation~\eqref{eq:charac_Eh} of the set $E_h$, $G(h,\bm{p})\subset E_h$, i.e. $\alpha_{X,t} = h$.

  We aim to prove that the 2-microlocal frontier of $X$ at $t$ shows a chirp oscillation behaviour: $\sigma_{X,t}(s') > h+s'$ for all $s'<-h$. Similarly to the proof of Theorem~\ref{th:2ml_levy}, we therefore investigate the regularity of the integral of $X$. In addition, we assume that $\beta \geq 1$, as the proof in the other case $\beta<1$ is completely similar.

  Let us set $u\in\R$, $\eps>0$ and $h'\eqdef h(1+\eps)>h$. There exist $m>0$ such that $2^{-(m+1)/h'} < \abs{t-u} \leq 2^{-m/h'}$. Furthermore, let $\ell\in\N$ be the greatest integer such that $j_{n(\ell)} \leq m$. We have to distinguish two different cases depending on the value of $m$. 

  Let us first suppose that $j_{n(\ell)}(1+\eps) \leq m$. Since $1/h' > \beta$, Lemma~\ref{lemma:tech_lemma2} implies that 
  \[
    \forall v\in B(t,\abs{u-t});\quad \abs{ X^m_v - X^m_t } \leq c \,m2^{-m} \leq c \,\log\pthb{\abs{u-t}^{-1}}\,\abs{ u-t }^{h'}.
  \]
  Furthermore, we note that $\abs{t-u} \leq 2^{-j_{n(\ell)}(1+\eps)/h'} = 2^{-j_{n(\ell)}/h}$, implying there is no jump time $s$ such that $\abs{\Delta X_s}\geq 2^{-m}$ and $s\in B(t,\abs{u-t})$. Using in addition the estimates on the drift obtained in Proposition~\ref{prop:pointwise_levy}, we obtain
  \[
    \forall v\in B(t,\abs{u-t});\quad \abs{ X_v - X_t } \leq c \,\log\pthb{\abs{u-t}^{-1}}\,\abs{ u-t }^{h'}.
  \]
  Therefore,
  \[
    \absbb{ \int_t^u \pth{X_v-X_t}\,\dt v } \leq c \,\log\pthb{\abs{u-t}^{-1}}\,\abs{ u-t }^{1+h'},
  \]
  where we recall that $h'>h$.

  We now consider the second case $j_{n(\ell)}\leq m \leq j_{n(\ell)}(1+\eps)$. As previously, we know that for all $v\in B(t,\abs{u-t})$, $\abs{ X^m_v - X^m_t } \leq c \,\log\pthb{\abs{u-t}^{-1}}\,\abs{ u-t }^{h'}$. Nevertheless, in this case, there might exist a \emph{double-jump} of size $2^{-j_{n(\ell)}}$ inside the interval $B(t,\abs{u-t})$. Owing to the construction of the set $G(h,\bm{p})$, the contribution of the \emph{double-jump} to $\abs{ \int_t^u \pth{X_v-X_t}\,\dt v }$ is upper-bounded by
  \begin{align*}
    2^{-j_{n(\ell)}} \cdot 2^{-j_{n(\ell)}\gamma} + 2^{-j_{n(\ell)}\rho} \,\abs{u-t}
    &\leq c\,\abs{u-t}^{h'(1+\gamma)/(1+\eps)} + c\, \abs{u-t}^{\rho\, h'/(1+\eps)+1} \\
    &= c\,\abs{u-t}^{h(1+\gamma)} + c\, \abs{u-t}^{\rho h+1}.
  \end{align*}
  In the previous exponents, we note that $\rho>1$ and $\gamma>1/h$, implying that $\rho h+1 > h+1$ and $h(1+\gamma) > h+1$. Hence, we have proved there exists $\eps_0>0$ such that
  \[
    \absbb{ \int_t^u \pth{X_v-X_t}\,\dt v } \leq c \,\abs{ u-t }^{1+h+\eps_0},
  \]
  for all $u$ in the neighbourhood of $t$.
  This last inequality proves that the regularity at $t$ is singular, as the 2-microlocal frontier must satisfy
  \begin{align*}
    \forall s'\leq -h;\quad \sigma_{X,t}(s') \geq \frac{s'+h}{1+\eps_0}.
  \end{align*}
  Therefore, with probability one, for all $h\in\ivoo{1/(\alpha_1+\alpha_2),1/\beta}$, $G(h,\bm{p})\subset \widehat{E}_h$, and $\dimH\,\widehat{E}_h \geq (\alpha_1+\alpha_2-\gamma+1-\rho)h$. Considering rational parameters such that $\gamma\rightarrow 1/h$ and $\rho\rightarrow 1$, we obtain $\dimH\,\widehat{E}_h \geq (\alpha_1+\alpha_2) h-1$. Finally, since the previous reasoning holds on any interval $\ivff{a,b}$ where $a,b\in\Q$, with probability one, we have proved the expected lower bound of the Hausdorff dimension:
  \[
    \forall V\in\Oi, \ \forall h\in\ivoo{1/(\alpha_1+\alpha_2)), 1/\beta};\quad  \dimH \pth{\widehat{E}_h\cap V} \geq (\alpha_1+\alpha_2) h-1.
  \]
\end{proof}


\section{Linear (multi)fractional stable motion}  \label{sec:2ml_lfsm}

The linear fractional stable motion (LFSM) is a stochastic process that has been considered by several authors: \citet{Maejima-1983,Takashima-1989,Kono.Maejima-1991,Samorodnitsky.Taqqu-1994,Ayache.Roueff.ea-2009,Ayache.Hamonier-2013}. Its general integral form is defined by
\begin{align} \label{eq:def_lfsm2}
  X_t = \int_\R \brcB{ & a^+ \bktb{ (t-u)_+^{H-1/\alpha} - (-u)_+^{H-1/\alpha} } \nonumber \\
  +& a^- \bktb{ (t-u)_-^{H-1/\alpha} - (-u)_-^{H-1/\alpha} } } \,M_{\alpha}(\dt u),
\end{align}
where $H\in\ivoo{0,1}$, $(a^+,a^-)\in\R^2\setminus(0,0)$ and $M_{\alpha}$ is an $\alpha$-stable random measure on $\R$ with Lebesgue control measure $\lambda$ and skewness intensity $\beta_\alpha(\cdot)\in\ivff{-1,1}$. Throughout this paper, it is assumed that $\beta_\alpha$ is constant, and equal to zero when $\alpha=1$. In this context, for any Borel set $A\subset\R$, the characteristic function of $M_{\alpha}(A)$ is given by
\[
  \espb{ e^{i\theta M_{\alpha}(A)} } = 
  \begin{cases}
    \exp\brcb{ -\lambda(A)\abs{\theta}^\alpha \pthb{ 1 - i\beta_\alpha \,\text{sign}(\theta)\tan(\alpha\pi/2) } } \quad &\text{if } \alpha\in\ivoo{0,1}\cup\ivoo{1,2}; \\
    \exp\brcb{ -\lambda(A)\abs{\theta} } &\text{if }\alpha = 1.
  \end{cases}
\]
For the sake of readability, we consider in the rest of the section the particular case $(a^+,a^-) = (1,0)$ (even though as stated \cite{Samorodnitsky.Taqqu-1994}, the law of the process depends on values $(a^+,a^-)$ chosen).

To begin with, we present in the next statement an alternative representation for the two-parameter field
$(t,H) \mapsto X(t,H) = \int_\R  \brcb{ (t-u)_+^{H-1/\alpha} - (-u)_+^{H-1/\alpha} }  \,M_{\alpha}(\dt u)$.
In the case $H\geq H/\alpha$, the formula has been previously obtained by \citet{Takashima-1989}.
\begin{proposition}  \label{prop:rep_lfsm}
 For all $t\in\R$ and $H\in\ivoo{0,1}$, the random variable $X(t,H)$ satisfies
 \begin{equation} \label{eq:lmsm_representation}
   X(t,H) \eqas
   \begin{cases}
     \displaystyle C_H\int_\R L_u \brcB{ (t-u)_+^{H-1/\alpha-1} - (-u)_+^{H-1/\alpha-1} } \,\dt u&\text{if } H\in\ivfob{\frac{1}{\alpha},1}; \\
     L_t &\text{if } H = \frac{1}{\alpha} \\
     \displaystyle C_H\int_\R \brcB{ (L_u - L_t) (t-u)_+^{H-1/\alpha-1} - L_u (-u)_+^{H-1/\alpha-1} } \,\dt u& \text{if } H\in\ivofb{0,\frac{1}{\alpha}},
   \end{cases}
 \end{equation}
 where $C_H = H - 1/\alpha$ and $L$ is an $\alpha$-stable L\'evy process defined by
 \[
   \forall t\in\R_+\quad L_t =  M_{\alpha}(\ivff{0,t})\quad\text{and}\quad \forall t\in\R_-\quad L_t = -M_{\alpha}(\ivff{t,0}).
 \]
\end{proposition}
\begin{proof}
  For the sake of readiness, we present in the proof for any $H\in\ivoo{0,1}$, even though the first case can be found in \cite{Takashima-1989}.
  Suppose $t\in\R$ and $H\in\ivoo{0,1}$. Since $(L_t)_{t\in\R}$ is an $\alpha$-stable L\'evy process, it has c\`adl\`ag sample paths. According to \cite{Applebaum-2009} (chap. 4.3.4), the theory of the stochastic integration based $\alpha$-stable L\'evy processes coincide integrals with respect to $\alpha$-stable random measure. Therefore, the r.v. $X(t,H)$ is almost surely equal to
  $\int_\R  \brcb{ (t-u)_+^{H-1/\alpha} - (-u)_+^{H-1/\alpha} }  \,\dt L_u$.
  Let $\eps>0$ and $b < t$. Using a classic integration by parts, we obtain
  \begin{align} \label{eq:int_part_stable}
    L_{t-\eps} \eps^{H-1/\alpha} - L_s (t-b)^{H-1/\alpha} &= \int_b^{t-\eps} (t-u)^{H-1/\alpha} \,\dt L_u \nonumber \\
    &- \pthB{H - \frac{1}{\alpha}} \int_b^{t-\eps} L_{u} (t - u)^{H-1/\alpha - 1} \,\dt u.
  \end{align}
  \begin{enumerate}[ \it 1.]
    \item If $H\in\ivoob{\frac{1}{\alpha},1}$, $H-1/\alpha > 0$. Hence, $\int_b^{t-\eps} L_{u} (t - u)^{H-1/\alpha - 1} \,\dt u$ almost surely converges to $\int_b^t L_{u-} (t - u)^{H-1/\alpha - 1} \,\dt u$ when $\eps\rightarrow 0$. Similarly, $\int_b^{t-\eps} (t-u)^{H-1/\alpha} \,\dt L_u$ converges in $L^\alpha(\Omega)$. Therefore, using Equation~\eqref{eq:int_part_stable} with $t=0$ and $b < 0$, we obtain almost surely
    \begin{align*}
      \int_b^t  \brcB{ (t-u)^{H-1/\alpha} - (-u)^{H-1/\alpha} }  \,\dt L_u 
      &= C_H\int_b^t L_u \brcB{ (t-u)^{H-1/\alpha-1} - (-u)^{H-1/\alpha-1} } \,\dt u \\
      &- L_b\brcB{ (t-b)^{H-1/\alpha} - (-b)^{H-1/\alpha} }.
    \end{align*}
    When $b\rightarrow-\infty$, the left-term clearly converges to $X(t,H)$ in $L^\alpha(\Omega)$.
    According to \cite{Pruitt-1981}, we know that almost surely for any $\eps>0$, $\limsup_{u\rightarrow -\infty} u^{1/\alpha+\eps} \abs{L_u} = 0$. Furthermore, we also have $(t-u)^{H-1/\alpha-1} - (-u)^{H-1/\alpha-1} \sim_{-\infty} (-u)^{H-1/\alpha-2}$ and $(t-b)^{H-1/\alpha} - (-b)^{H-1/\alpha} \sim_{-\infty} (-b)^{H-1/\alpha-1}$. Therefore, as $H < 1$ and using the dominated convergence theorem, the right-term almost surely converges to the expected integral.\vsp
    
    \item If $H\in\ivoob{0,\frac{1}{\alpha}}$, we observe that Equation~\eqref{eq:int_part_stable} can be slightly transformed into
    \begin{align*}
      &(L_{t-\eps} -L_t)\eps^{H-1/\alpha} - (L_b-L_t) (t-b)^{H-1/\alpha} \\
      &= \int_b^{t-\eps} (t-u)^{H-1/\alpha} \,\dt L_u - \pthB{H - \frac{1}{\alpha}} \int_b^{t-\eps} (L_{u}-L_t) (t - u)^{H-1/\alpha - 1} \,\dt u.
    \end{align*}
    According to \cite{Pruitt-1981}, $\alpha_{Y,t} \eqas 1/\alpha$. Therefore, up to an extracted sequence, the previous expression almost surely converges when $\eps\rightarrow 0$ and using a similar formula for $t=0$, we obtain
    \begin{align*}
      &\int_b^t  \brcB{ (t-u)^{H-1/\alpha} - (-u)^{H-1/\alpha} }  \,\dt L_u \\
      &= C_H\int_b^t \brcB{ (L_u - L_t) (t-u)^{H-1/\alpha-1} - L_u (-u)^{H-1/\alpha-1} } \,\dt u \\
      &- L_b\brcB{ (t-b)^{H-1/\alpha} - (-b)^{H-1/\alpha} } + L_t (t-b)^{H-1/\alpha}.
    \end{align*}
    The property $\limsup_{u\rightarrow -\infty} u^{1/\alpha+\eps} \abs{L_u} = 0$ and the previous equivalents finally prove Equation~\eqref{eq:lmsm_representation}.
  \end{enumerate}
  To end this proof, let us consider the integral representation in the particular case $H=1/\alpha$. In fact, Equation~\eqref{eq:lmsm_representation} is a slightly misuse since the expression does not exist. Nevertheless, we prove that it converges almost surely to $X(t,1/\alpha) = L_t$ when $H\rightarrow 1/\alpha$. 
  
  Suppose first that $H\nearrow 1/\alpha$ and rewrite $X(t,H)$ as
  \[
    X(t,H) = C_H\int_\R \brcB{ (L_u - L_t\indi_{u\geq b}) (t-u)_+^{H-1/\alpha-1} - L_u (-u)_+^{H-1/\alpha-1} } \,\dt u + L_t(t-b)^{H-1/\alpha},
  \]
  The first component of the expression converges to zero since $C_H\rightarrow_{H\rightarrow 1/\alpha} 0$ and $\alpha_{Y,t} \eqas 1/\alpha$. As the second part simply converges to $L_t$, we get the expected limit.
  The case $H\searrow 1/\alpha$ is treated similarly.
\end{proof}
Note that \citet{Picard-2011} has determined a similar representation for fractional Brownian motion.

\begin{proof}[Proof of Theorem \ref{th:2ml_lfsm}]
  Let us set $H\in\ivoo{0,1}$ and $\alpha\in\ivfo{1,2}$. In order to obtain the multifractal structure of the LFSM, we first relate the 2-microlocal frontier of $X$ at $t$ to the frontier of the alpha-stable process $L$. 
  \begin{enumerate}[ \it 1.]
    \item If $H > 1/\alpha$, we note that the representation obtained in Proposition \ref{prop:rep_lfsm} is defined almost surely for all $t\in\R$. Therefore, let us set $\omega\in\Omega$ and $t\in\R$. As previously, we can assume that $t\in\ivff{0,1}$. Then,  
    \begin{align*}
      X_t &= C_H\int_b^t L_u (t-u)_+^{H-1/\alpha-1} \,\dt u + C_H\int_b^0 L_u (-u)_+^{H-1/\alpha-1} \,\dt u \\
      &+ C_H\int_{-\infty}^b L_u \brcB{ (t-u)_+^{H-1/\alpha-1} - (-u)_+^{H-1/\alpha-1} } \,\dt u,
    \end{align*}
    where $b < 0$ is fixed. The second term is simply a constant that does not influence the regularity. Similarly, using the dominated convergence theorem, we note that the third one is a smooth function on the interval $\ivff{0,1}$, and therefore has no impact on the 2-microlocal frontier. 

    Therefore, we only need to focus on the first term. Let us define the process $Y_u = L_u\indi_{\brc{u\geq t}}$. Since the 2-microlocal spaces and frontier presented in Definitions~\ref{def:2ml_spaces} and \ref{def:2ml_spaces_gen} are localised at a point $t$, we necessarily have $\sigma_{L,t} = \sigma_{Y,t}$. Furthermore, we note that
    \[
      Z_t \eqdef C_H \int_b^t L_u (t-u)_+^{H-1/\alpha-1} \,\dt u = C_H \int_\R Y_u (t-u)_+^{H-1/\alpha-1} \,\dt u = \pthb{I^{H-1/\alpha}_+ Y}(t)
    \]
    Owing to the property of stability of 2-microlocal spaces under fractional integration (see see Theorem 1.1 in \cite{Jaffard.Meyer-2000}), we obtain $\sigma_{Z,t} = \sigma_{Y,t}+H-1/\alpha$, and therefore $\sigma_{X,t} = \sigma_{L,t}+H-1/\alpha$.\vsp
    
    \item If $H < 1/\alpha$, we first observe that according to the multifractal spectrum of alpha-stable processes and $H > 0$, $\dimH(\brc{ t\in\R : \alpha_{L,t} \leq 1/\alpha - H })  < 1$. Hence, for almost every $\omega\in\Omega$, Formula~\eqref{eq:lmsm_representation} is well-defined almost everywhere on $\R$. Anywhere else, we may simply assume that $X(t,H)$ is set to zero. We will explain later why the value $0$ at these particular times does not modify the 2-microlocal frontier.

    Similarly to the previous case $H>1/\alpha$, the regularity of $X$ only depends on the behaviour of the component
    \[
      Z:t\longmapsto C_H \int_b^t (L_u - L_t) (t-u)_+^{H-1/\alpha-1}\,\dt u.
    \]
    One might recognize a Marchaud fractional derivative (see e.g. \cite{Samko.Kilbas.ea-1993}). Let us modify this expression to exhibit a more classic form of fractional derivative. For almost all $s\in\ivff{0,1}$ and $\eps>0$, we have
    \begin{align*}
      \int_b^{s-\eps} L_u (s-u)^{H-1/\alpha}\,\dt u
      &= C_H \int_b^{s-\eps} L_u \,\dt u \int_{u+\eps}^s (v-u)^{H-1/\alpha-1}\,\dt v + \eps^{H-1/\alpha} \int_b^{s-\eps} L_u\,\dt u \\
      &= C_H \int_b^{s-\eps} \dt u \int_{u+\eps}^s (L_u - L_v)(v-u)^{H-1/\alpha-1}\,\dt v  \\
      &+ \eps^{H-1/\alpha} \int_b^{s-\eps} L_u\,\dt u+ C_H \int_b^{s-\eps} \dt u \int_{u+\eps}^s L_v(v-u)^{H-1/\alpha-1}\,\dt v 
    \end{align*}
    The last two terms are equal to
    \[
      \eps^{H-1/\alpha} \int_b^{s-\eps} L_u\,\dt u - \eps^{H-1/\alpha} \int_{b+\eps}^s L_v\,\dt v + \int_{b+\eps}^s L_v (v-b)^{H-1/\alpha} \,\dt v,
    \]
    which converges to $\int_b^s L_v (v-b)^{H-1/\alpha} \dt v$ as $\eps\rightarrow 0$ since $H-1/\alpha>-1$. Similarly, as we consider times at which $\alpha_{L,s} > 1/\alpha - H$, the dominated convergence theorem implies that the first term converges to $C_H \int_b^s \dt v \int_b^v (L_u - L_v)(v-u)^{H-1/\alpha-1}\,\dt u$. Therefore,
    \[
      \int_b^s L_u (s-u)^{H-1/\alpha}\,\dt u = C_H \int_b^s \dt v \int_b^v (L_u - L_v)(v-u)^{H-1/\alpha-1}\,\dt u + \int_b^s L_v (v-b)^{H-1/\alpha} \,\dt v.
    \]
    According to classic real analysis results, the previous expression is differentiable almost everywhere on the interval \ivff{0,1}, and therefore
    \[
      Z_t \overset{\text{a.e.}}{=} \frac{\dt}{\dt t} \int_b^t L_u (t-u)^{H-1/\alpha}\,\dt u -  L_t (t-b)^{H-1/\alpha},
    \]
    for almost all $t\in\ivff{0,1}$. Note that the last two formulas ensure that $Z_\sbullet\in L^1_{\text{loc}}(\R)$, and thus $X_\sbullet\in L^1_{\text{loc}}(\R)$ with probability one.\vsp

    Let us now explain in which sense we investigate the 2-microlocal regularity of $X$. As previously outlined in the introduction, in the case $H<1/\alpha$, sample paths of LFSM are nowhere bounded. As a consequence, it is meaningless to consider the usual Hölder regularity. On the other hand, the 2-microlocal formalism has been introduced in a more general frame which are distributions $\Di'(\R)$. Since we have previously proved that $X_\sbullet\in L^1_{\text{loc}}(\R)$, with probability one, $X_\sbullet$ is a distribution whose 2-microlocal frontier is well-defined. We refer to \cite{Meyer-1998} for a complete presentation of the 2-microlocal spaces for distributions. Also note that in this context, we can modify the values of $X_t$ on the negligible set $\brc{ t\in\R : \alpha_{L,t} \leq 1/\alpha - H }$ without modifying $X_\sbullet$ in the sense of distributions.

    Then, let first consider the term $Y:t\mapsto \int_b^t L_u (t-u)^{H-1/\alpha}\,\dt u$. Since $H-1/\alpha>-1$, it is a Riemann--Liouville fractional integral of order $H-1/\alpha+1>0$.  Hence, using the techniques previously presented, we obtain that $\sigma_{Y,t} = \sigma_{L,t}+H-1/\alpha+1$. Furthermore, the almost everywhere derivative $\frac{\dt}{\dt t} \int_b^t L_u (t-u)^{H-1/\alpha}\,\dt u$ coincide with the derivative in the sense distribution. Still using the stability of 2-microlocal spaces, the 2-microlocal frontier of the latter is therefore equal to $\sigma_{L,t}+H-1/\alpha$. In addition, the 2-microlocal frontier of $t\mapsto L_t (t-b)^{H-1/\alpha}$ is equal to $\sigma_{L,t}$ (the multiplication with a locally smooth function having no effect). Hence, as $\sigma_{L,t} > \sigma_{L,t}+H-1/\alpha$, we have proved that $\sigma_{Z,t}=\sigma_{L,t}+H-1/\alpha$, and thus $\sigma_{X,t}=\sigma_{L,t}+H-1/\alpha$ with probability one.
  \end{enumerate}
  Therefore, in both cases, we have proved that with probability one and for all $t\in\ivff{0,1}$,
  \[
    \forall s'\in\R;\quad \sigma_{X,t}(s')=\sigma_{L,t}(s')+H-1/\alpha.
  \]
  Then, using the same reasoning as in the proof of Corollary~\ref{cor:2ml_levy}, we observe that
  \begin{equation*}
    \forall s\in\ivff{0,1/\alpha};\quad \widetilde{E}_{s-H+1/\alpha}(L)  \subseteq E_{\sigma,s'}(X) \subseteq \widetilde{E}_{s-H+1/\alpha}(L) \cup \bigcup_{h<s-H+1/\alpha} \widehat{E}_{h}(L).
  \end{equation*}
  for any $\sigma < H-1/\alpha$. Furthermore, since $:$ is an alpha-stable process, Theorem~\ref{th:2ml_levy} induces that 
  \[
    \dimH\, \widetilde{E}_{s-H+1/\alpha}(L) = \alpha\pthb{s-H+1/\alpha} = \alpha(s-H)+1
  \]
  and $\dimH\,\widehat{E}_{h}(L) < \alpha(s+H)-1$ for every $h<s-H+1/\alpha$. These two estimates clearly prove the spectrum presented in Equation~\eqref{eq:2ml_spectrum_lfsm}.
  Finally, the spectrum of singularity for the weak scaling exponent is obtained similarly.
\end{proof}

Another class of processes similar to the LFSM has been introduced and studied in \cite{Benassi.Cohen.ea-2004,Marquardt-2006,Cohen.Lacaux.ea-2008}. Named \emph{fractional L\'evy processes}, it is defined by
\[
  X_t = \frac{1}{\Gamma(d+1)} \int_\R \brcB{ (t-u)_+^d - (-u)_+^d }\, L(\dt u),
\]
where $d\in\ivoo{0,1/2}$ and $L$ is a L\'evy process enjoying $Q=0$ (no Brownian component), $\esp{L(1)} = 0$ and $\esp{L(1)^2} < +\infty$. Owing to this last assumption on $L$, LFSMs are not fractional L\'evy processes. Nevertheless, their multifractal regularity can be determined as well.
\begin{proposition}  \label{prop:2ml_flp}
  Suppose $X$ is a fractional L\'evy process parametrized by $d\in\ivoo{0,1/2}$. Then, with probability one and for all $\sigma\leq d$,
  \begin{equation} \label{eq:2ml_spectrum_flp}
    \forall V\in\Oi;\quad \dimH\pth{ E_{\sigma,s'}\cap V } =
    \begin{cases}
      \beta (s - d)   & \text{ if } s\in\ivffb{d,d+\tfrac{1}{\beta}}; \\
      -\infty   & \text{ otherwise.}
   \end{cases}
  \end{equation}
  where $\beta$ designates the Blumenthal--Getoor exponent of the L\'evy process $L$. Furthermore, for all $s'\in\R$, $E_{\sigma,s'}$ is empty if $\sigma > d$. 
\end{proposition}
\begin{proof}
  \citet{Marquardt-2006} has established (Theorem 3.4) a representation of fractional L\'evy processes equivalent to Proposition~\ref{prop:rep_lfsm}: 
  \[
    X_t = \frac{1}{\Gamma(d)} \int_\R L_u \brcB{ (t-u)_+^{d-1} - (-u)_+^{d-1} }\, \dt u.
  \]
  Based on this result, a straightforward adaptation of the proof of Theorem~\ref{th:2ml_lfsm} yields Equation~ \eqref{eq:2ml_spectrum_flp}.
\end{proof}
Similarly to the LFSM, this statement refines regularity results established in \cite{Benassi.Cohen.ea-2004,Marquardt-2006} and proves that the multifractal spectrum of a fractional L\'evy process is equal to
\begin{equation} \label{eq:spectrum_flp}
  \forall V\in\Oi;\quad d_X(h,V) = 
  \begin{cases}
    \beta \pth{ h - d }   & \text{ if } h\in\ivffb{d,d+\tfrac{1}{\beta}}; \\
    -\infty   & \text{ otherwise}.
  \end{cases}
\end{equation}

Let us finally conclude this section with the proof of Theorem \ref{th:2ml_lmsm}.
\begin{proof}[Proof of Theorem \ref{th:2ml_lmsm}]
  Suppose $(X_t)_{t\in\R}$ is a linear multifractional stable motion with $\alpha\in\ivoo{1,2}$ and Hurst function $H(\cdot)\in\ivoo{1/\alpha,1}$. According to the representation obtained in Proposition~\ref{prop:rep_lfsm}, $X_t$ is almost surely equal to $X(t,H(t))$.

  To begin with, we first use the uniform estimate of the local Hölder exponent obtained by \citet[Th. 8.1]{Ayache.Hamonier-2013} to obtain an upper bound on the 2-microlocal frontier. The latter have proved that with probability one and for all $t\in\R$, $\widetilde{\alpha}_{X,t} = H(t)-1/\alpha$.  In addition, the 2-microlocal frontier is known to satisfy the inequality $\sigma_{X,t} \leq \liminf_{u\rightarrow t} \widetilde{\alpha}_{X,u}$ for any $t\in\R$, which proves that $\sigma_{X,t}\leq H(t)-1/\alpha$ with probability one.

  Let us now set $\omega\in\Omega$ and $t\in\R$ and decompose $X$ into two parts:
  \[
    X_u = X(u,H(t)) + \pthb{ X(u,H(u)) - X(u,H(t)) }.
  \]
  According to the proof of Theorem~\ref{th:2ml_lfsm}, we already that the 2-microlocal frontier of the first component $X(\sbullet,H(t))$ is equal to $\sigma_{L,t}+H(t)-1/\alpha$. As a consequence, we have to prove that the second term $Y_u\eqdef X(u,H(u)) - X(u,H(t))$ is negligible in terms of 2-microlocal regularity. For that purpose, we observe that for any $u,v\in B(t,\rho)$
  \begin{align*}
    Y_u - Y_v 
    &= X(u,H(u)) - X(u,H(t)) - X(v,H(v)) + X(v,H(t)) \\
    &= \int_{H(t)}^{H(u)} \partial_H X(u,h) \,\dt h - \int_{H(t)}^{H(v)} \partial_H X(v,h) \,\dt h.
  \end{align*}
  Therefore, since $H$ is $\delta$-Hölderian,
  \begin{align*}
    \abs{ Y_u-Y_v } 
    &\leq \int_{H(t)}^{H(u)} \absb{ \partial_H X(u,h) - \partial_H X(v,h) } \,\dt h + \int_{H(u)}^{H(v)} \absb{ \partial_H X(v,h) } \,\dt h \\
    &\leq c\,\abs{u-t}^\delta\cdot\abs{u-v}^\gamma + c\,\abs{u-v}^\delta
  \end{align*}
  where $\gamma < \inf_{u\in B(t,\rho)} H(u)-1/\alpha$ and $\delta > \sup_{u\in\R} H(u)$. Using Definition~\ref{def:2ml_spaces} of the 2-microlocal spaces, this inequality proves that 
  $\sigma_{Y,t} \geq \pthb{ \delta+s' }\wedge\pthb{H-1/\alpha}$ for all $s'\in\R$.
  Since $\delta>H(t)$, $\sigma_{X,t}\leq H(t)-1/\alpha$ and $\sigma_{L,t}(s') \leq (1/\alpha+s')\wedge 0$, we therefore obtain with probability one and for all $t\in\R$
  \[
    \forall s'\in\R;\quad \sigma_{X,t}(s') = \sigma_{L,t}(s')+H(t)-1/\alpha.
  \]
  Then, the multifractal structure described in Equation~\eqref{eq:2ml_spectrum_lmsm} is obtained using the same arguments as in the proof of Theorem~\ref{th:2ml_lfsm}.
\end{proof}

  
  

\begin{remark}
  In the case $H(\cdot)$ does not satisfy the assumption $\delta > \sup_{t\in\R} H(t)$, the proof of Theorem \ref{th:2ml_lmsm} can be modified to extend the statement and generalize results obtained in \cite{Stoev.Taqqu-2005}. This complete study is made in \cite{Balanca.Herbin-2013} for the multifractional Brownian motion. For the sake of clarity, we prefer to focus in this work on $(\Hi_0)$-Hurst functions and the multifractal structure of the LMSM presented in Theorem~\ref{th:2ml_lmsm}.
\end{remark}

\begin{remark}
  Even though it is assumed all along this section that $H(\cdot)$ is deterministic, owing to the deterministic representation presented in Proposition \ref{prop:rep_lfsm}, Theorems \ref{th:2ml_lfsm} and \ref{th:2ml_lmsm} still hold if $H(\cdot)$ is a continuous random process. Hence, based on these results, a class of random processes with random and non-homogeneous multifractal spectrum can be easily constructed. A similar extension of the multifractional Brownian motion has been introduced and studied by \citet{Ayache.Taqqu-2005}.
\end{remark}



\end{document}